\documentclass[10pt]{amsart}
\usepackage{amsmath}
\usepackage{amsthm}
\usepackage{amssymb}
\usepackage{amscd}
\usepackage{amsfonts}
\usepackage{graphics}
\usepackage[dvips]{epsfig}
\usepackage{graphicx}
\date{}

\newlength{\defbaselineskip}
\newcommand{\setlinespacing}[1]
           {\setlength{\baselineskip}{#1 \defbaselineskip}}
         
\theoremstyle{plain}
\newtheorem{theorem}{Theorem}[section]
\newtheorem{proposition}[theorem]{Proposition}
\newtheorem{lemma}[theorem]{Lemma}

\newtheorem{corollary}[theorem]{Corollary}

\theoremstyle{definition}
\newtheorem{definition}[theorem]{Definition}

\newtheorem{remark}[theorem]{Remark}

\theoremstyle{remark}

\newcommand{\na}{\mathbb{N}}

\newcommand{\re}{\mathbb{R}}

\newcommand{\LL}{\>\hbox{\vrule width.2pt \vbox to 7pt{\vfill\hrule width 7pt height.2pt}}\>}

\catcode`\@=11

\def\@ceqnnum{{\reset@font\rm (\tempequation)}}

\def\@ceqncr{{\ifnum0=`}\fi\@ifstar{\global\@eqpen\@M
    \@cyeqncr}{\global\@eqpen\interdisplaylinepenalty \@cyeqncr}}

\def\@cyeqncr{\@ifnextchar [{\@cxeqncr}{\@cxeqncr[\z@]}}

\def\@cxeqncr[#1]{\ifnum0=`{\fi}\@@ceqncr
   \noalign{\penalty\@eqpen\vskip\jot\vskip #1\relax}}

\def\@@ceqncr{\let\@tempa\relax
    \ifcase\@eqcnt \def\@tempa{& & &}\or \def\@tempa{& &}%
      \else \def\@tempa{&}\fi
     \@tempa \if@eqnsw\@ceqnnum\fi
     \global\@eqnswtrue\global\@eqcnt\z@\cr}

\def\clabel#1#2{\global\let\tempequation #2
   \@bsphack\if@filesw {\let\thepage\relax
   \def\protect{\noexpand\noexpand\noexpand}%
   \edef\@tempa{\write\@auxout{\string
      \newlabel{#1}{{#2}{\thepage}}}}%
   \expandafter}\@tempa
   \if@nobreak \ifvmode\nobreak\fi\fi\fi\@esphack}

\def\ceqnarray{
\global\@eqnswtrue\m@th
\global\@eqcnt\z@\tabskip\@centering\let\\\@ceqncr
$$\halign to\displaywidth\bgroup\@eqnsel\hskip\@centering
  $\displaystyle\tabskip\z@{{}##}$&\global\@eqcnt\@ne
  \hskip 2\arraycolsep \hfil${{}##}$\hfil
  &\global\@eqcnt\tw@ \hskip 2\arraycolsep
$\displaystyle\tabskip\z@{{}##}$\hfil
   \tabskip\@centering&\llap{##}\tabskip\z@\cr}

\def\endceqnarray{\@@ceqncr\egroup$$\global\@ignoretrue}

\catcode`\@=12

\long\def\salta#1{\relax}

\def\rn{\mathbb{R}^{N}}

\def\liq{L^{\infty}(Q)}

\def\be{\begin{equation}}
\def\ee{\end{equation}}

\def\rife#1{(\ref{#1})}
\def\rifer#1{({\rm \ref{#1}})}
\def\a#1{a(t,x,\nabla #1)}

\def\hnve{H_{n}(v^\vare )}

\def\hnv{H_{n}(v )}

\def\ohnv{\overline{H}_{n}(v )}
\def\tkve{T_{k}(v^\vare)}
\def\ve{v^\vare}
\def\ue{u^\vare}
\def\ut{\tilde{u}}
\def\tkv{T_{k}(v )}

\def\tkvn{T_{k}(v )_\nu}

\def\vare{\varepsilon}

\def\t1p0{T^{1,p}_{0}(\Omega)}

\def\capp{\text{\rm{cap}}_{p}}
\def\m2{M^{\frac{N(p-1)}{N-1}}(\Omega)}

\def\div{{\rm div}}
\def\regq{C^{\infty}_{0}(Q)}

\def\sob{W^{1,p}_{0}(\Omega)}
\def\mhe{\hat{\mu}^{\vare}_{0}}

\def\msp{\mu_{s}^{+}}
\def\msm{\mu_{s}^{-}}
\def\lep{\lambda_{\oplus}^{\vare}}
\def\lem{\lambda_{\ominus}^{\vare}}
\def\mh{\hat{\mu}_{0}}
\def\ms{\mu_{s}}

\def\into{\int_{\Omega}}
\def\intq{\int_{Q}}

\def\intt{\int_{0}^{T}}
\def\liq{L^{\infty}(Q)}
\def\w-1p'{W^{-1,p'}(\Omega)}
\def\pw-1p'{L^{p'}(0,T;W^{-1,p'}(\Omega))}
\def\l{\textsl{L}}
\def\C{C^{\infty}_{0}(Q)}

\def\dys{\displaystyle}

\def\luq{L^{1}(Q)}
\def\lp'n{(L^{p'}(\Omega))^{N}}

\def\supp{\text{\rm{supp}}}
\def\fdp{\varphi^{+}_{\delta}}
\def\fdm{\varphi^{-}_{\delta}}
\def\fde{\Phi_{\delta,\eta}}

\def\pdp{\psi_{\delta}^{+}}
\def\pdm{\psi_{\delta}^{-}}
\def\pep{\psi_{\eta}^{+}}
\def\pem{\psi_{\eta}^{-}}
\def\benve{\beta_{n}(\ve)}
\def\pdep{\psi_{\delta}^{+}\psi_{\eta}^{+}}
\def\pdem{\psi_{\delta}^{-}\psi_{\eta}^{-}}
\def\kdp{K^{+}_{\delta}}
\def\kdm{K^{-}_{\delta}}
\def\ohp{\overline{H}(\varphi^{+}_{\delta})}
\def\ohm{\overline{H}(\varphi^{-}_{\delta})}
\def\hp{H(\varphi^{+}_{\delta})}

\def\udp{U^{+}_{\delta}}
\def\udm{U^{-}_{\delta}}

\def\pep{\psi_{\eta}^{+}}
\def\pem{\psi_{\eta}^{-}}

\def\fde{\Phi_{\delta,\eta}}

\def\regot{C^{\infty}_{0}([0,T]\times\Omega)}

\def\psob{L^{p}(0,T;W^{1,p}_{0}(\Omega))}
\def\plq{L^{q}(0,T;W^{1,q}_{0}(\Omega))}
\def\pli{L^{\infty}(0,T;L^1 (\Omega))}
\def\lild{L^{\infty}(0,T;L^2 (\Omega))}
\def\pd{\psi_{\delta}}
\def\sobl{L^2 (0,T ; H^1_0 ( \Omega ))}
\def\l2h-1{L^2 (0,T ; H^{-1} ( \Omega ))}
\def\uo{u_{0}}
\def\uoe{u_{0}^{\vare}}
\def\luo{L^{1}(\Omega)}
\def\lio{L^{\infty}(\Omega)}
\newcommand{\elle}[1]{L^{#1}(\Omega)}
\newcommand{\pelle}[1]{L^{#1}(Q)}

\def\om{\Omega}

\def\pbp#1#2#3#4{
\left\{\begin{array}{ll}
#1 & \mbox{in $(0,T)\times\Omega$,} \\
#2 & \mbox{on $(0,T)\times\partial\Omega$,} \\
#3 & \mbox{in $\Omega$#4}
\end{array} \right.}

\def\paraduale{L^{p'}(0,T;\duale)}

\def\duale{W^{-1,p'}(\Omega)}
\def\parelle#1{L^{#1}(Q)}

\def\car#1{\raise2pt\hbox{$\chi$}_{#1}}

\makeindex

\begin{document}

\setlinespacing{1}

\keywords{Nonlinear parabolic equations, parabolic capacity, measure data}
\subjclass[2000]{35B40, 35K55}

\author[F. Petitta]{Francesco Petitta}
\email{francesco.petitta@sbai.uniroma1.it}

\address[F. Petitta]{Dipartimento di Scienze di Base e Applicate
per l' Ingegneria, ``Sapienza", Universit\`a di Roma, Via Scarpa 16, 00161 Roma, Italy.}

\title[Renormalized solutions, parabolic problems with measures]{Renormalized solutions of nonlinear parabolic equations with general measure data}

\begin{abstract}

Let $\Omega\subseteq \rn$ a bounded open set, $N\geq 2$, and let
$p>1$; we prove existence of a renormalized solution for parabolic
problems whose model is
$$
\begin{cases}
    u_{t}-\Delta_{p} u=\mu & \text{in}\ (0,T)\times\Omega,\\
    u(0,x)=\uo  & \text{in}\ \Omega,\\
 u(t,x)=0 &\text{on}\ (0,T)\times\partial\Omega,
  \end{cases}
$$
where $T>0$ is any positive constant,  $\mu\in M(Q)$ is a any
measure with bounded variation over
 $Q=(0,T)\times\Omega$, and $\uo\in \luo$, and  $-\Delta_{p} u=-\div (|\nabla u|^{p-2}\nabla u )$ is the usual $p$-laplacian.
\end{abstract}
\keywords{Nonlinear parabolic equations, parabolic capacity, measure data}
\subjclass[2000]{35K55,  35D05, 35D10, 35R05}
\maketitle

\tableofcontents

\setcounter{equation}{0}
\section{Introduction}

Let $\Omega$ be a bounded, open subset of $\rn$, $T$ a positive number and
$Q=(0,T)\times\Omega$. We will deal with parabolic initial boundary value problems
\begin{equation}
\pbp{u_t+A(u) = \mu}{u = 0}{u(0)=u_0},
\label{pbase}
\end{equation}
where  $A$ is a nonlinear 
monotone and coercive operator in divergence form which acts from the 
$\psob$  into its dual $\paraduale$, $\mu$ is measure on $Q$ with bounded total variation, and $u_0\in \luo$.

As  a model example,
problem \rife{pbase} includes  the $p$-Laplace evolution equation:
\be\label{plap} \pbp{u_t-\div(|\nabla u|^{p-2}\nabla u)=\mu}{u = 0}{u(0)=u_0}.
\ee

As we said before, we are interested in the study of problem \rife{pbase} with a general Radon measure $\mu$ with bounded total variation on $Q$, and initial datum  $u_0\in\luo$ as data.
 Under our assumptions, if  $\mu\in\parelle{p'}$ and
$u_0\in L^2(\om)$, then problem \rife{pbase} turns out to have a  unique solution $u\in C (0,T;\elle2)$
in the weak sense (see
\cite{l}).
\salta{
 that is
\[
-\into \uo\ \varphi(0)\ dx-\int_{0}^{T}\langle \varphi_{t}, u  \rangle\ dt+\intq
a(t,x,\nabla u)\cdot\nabla \varphi\ dxdt=\int_{0}^{T}\langle f
,\varphi\rangle_{\w-1p',\sob}\ dt,
\]
for all $\varphi\in W $ such that $\varphi(T)=0$ (see
\cite{l}).
}

Under the  general assumption that $\mu$ and $u_0$ are bounded measures, the
existence of a distributional solution was proved in \cite{bdgo}, by
approximating \rife{pbase} with   problems having regular data and using
compactness arguments.
But, due to the lack of regularity of the solutions, the
distributional formulation is not strong enough
to provide uniqueness, as it can be proved by adapting to the parabolic case the counterexample of
J. Serrin  for the stationary problem (see \cite{se}, and refinements in \cite{p1}).\medskip

 In the case of linear operators the lack of uniqueness can be overcome by
defining the solution in a duality sense, and then  adapting the techniques of the stationary case introduced in \cite{s} (see also \cite{pe3}).

 However, for nonlinear operators  a new concept of solution was necessary 
 to get a  well--posed problem. In case of problem \rife{pbase}
with $\mu \in \luq$, this was done  independently in 
\cite{blm}     and in \cite{p} (see also \cite{amst}),  where the notions of 
{\sl renormalized solution}, and of {\sl entropy solution},  were respectively  
introduced. Both these approaches allow to obtain existence and uniqueness
 of solutions  if $\mu\in L^1(Q)$ and $u_0\in \elle1$. 
Unfortunately, these definitions do not extend directly to the case of a general, possibly singular, measure $\mu$.

 In \cite{dpp} the authors extend the result of existence and uniqueness
to a  larger class of measures which includes the $L^1$ case. Precisely, they
prove (in the framework of renormalized  solutions) that problem \rife{pbase}
has  a  unique solution  for every $u_0$ in $\elle1$ and for every measure $\mu$ which does not
charge the sets of null parabolic $p$-capacity (see Section \ref{c1} for the definition).

The importance of the
measures not charging sets of null $p$-capacity was first observed in the
stationary case in \cite{bgo}.

In order to use  a similar approach in the evolution case, 
the theory of $p$-capacity related  to the parabolic  operator $u_t+A(u)$ has been developed  \cite{dpp}, where the authors also 
investigated the relationships between  time--space dependent measures and
capacity.

\medskip
Thanks to a decomposition result proved in \cite{dpp} (see  Theorem \ref{cap2} below), if $\mu$ is
\emph{absolutely continuous} 
with respect to the $p$-capacity (these are usually called soft measures) one can still set  problem
\rife{pbase} in the framework of renormalized  solutions; as in the elliptic case, the idea formally consists in the use of test functions which depend on the solution itself.  
Thus, the definition of renormalized solution of problem \rife{pbase} can be extended to the case of general measure $\mu$ by adapting the idea of \cite{dmop} for the elliptic case.

Notice that the notion of renormalized solution and entropy solution for parabolic problem \rife{pbase} turn out to be equivalent as proved in \cite{dpr};  here we  extend the notion of renormalized solution for general measure data $\mu$ and so, thanks to this result, this notion will turn out to be coherent with all definitions of solution given before for problem \rife{pbase}.\medskip

The plan of the paper is as follows. We first introduce some basic knowledge on parabolic $p$-capacity and our main assumptions on the operator we deal with (Section \ref{c1} and Section \ref{2}). Then, in Section \ref{3}, we give the definition of renormalized solution for problem \rife{pbase},  deriving a useful  estimate enjoyed by any renormalized solution, and we state our main result on the existence of such a solution; in Section \ref{4} we prove that any renormalized solution of problem \rife{pbase} (actually a \emph{regular translation} of it), with possibly singular datum $\mu$, admits a $\capp$-quasi continuous representative defined $\capp$-quasi everywhere.

 In Section \ref{5} will see how the definition of renormalized solution does not depend on the decomposition (not uniquely determined) of the regular part of $\mu$ we mentioned above and to the statement of standard approximation argument we will use later, while  Section \ref{6} will be devoted to the proof of a key result, namely the strong convergence of the truncates of the approximating sequence in the energy space $\psob$. 

In Section \ref{7}, using the above result, we prove that there exists a renormalized solution of problem \rife{pbase}, and, finally, in Section \ref{8} we  deal with a partial uniqueness result (essentially for the linear case) and we state,   as an easy consequence of the definition of renormalized solution, 
 a sort of \emph{inverse maximum principle} for general parabolic operators.

\setcounter{equation}{0}
\section{About capacity}\label{c1}

 We  recall the notion of \emph{parabolic $p$-capacity} associated
to our problem (for further details see for instance
\cite{pierre}, \cite{dpp}).

\begin{definition}\label{cappar}
Let $Q=Q_{T}=(0,T)\times\Omega$ for any fixed $T>0$, and let us
call $V=\sob \cap L^{2}(\Omega)$ endowed with its natural norm
$\|\cdot\|_{\sob}+\|\cdot\|_{L^{2}(\Omega)}$ and
\[
W=\left\{ u\in L^{p}(0,T;V),\ u_{t}\in L^{p'}(0,T;V') \right\},
\]
endowed with its natural norm $\|u\|_{W}=\|u\|_{ L^{p}(0,T;V)}+
\|u_{t}\|_{ L^{p'}(0,T;V')}$. So, if $U\subseteq Q$ is an open set, we
define the \emph{parabolic $p$-capacity} of $U$ as
\[
\capp(U)=\inf\{\|u\|_{W}:u\in W,u\geq \chi_{U}\  \text{a.e. in}\  Q\},
\]
where as usual we set $\inf\emptyset=+\infty$. For any Borel
set $B\subseteq Q$ we then define
\[
\capp(B)=\inf\{\capp(U), U \ \text{open set of}\ Q, B\subseteq U\}.
\]
\end{definition}

Let us denote with $M(Q)$ the set of all Radon measures with bounded variation  on $Q$, while $M_{0}(Q)$ will denote the set of all measures with
bounded variation over $Q$ that do not charge the sets of
zero  $p$-capacity, that is if $\mu\in M_{0}(Q)$, then
$\mu(E)=0$, for all $E\subseteq Q$ such that $\capp(E)=0$.

\begin{remark}\label{rdefi2}
In \cite{dpp} the authors give another notion of parabolic capacity, equivalent to the one given here as far as sets of zero capacity are concerned; this different notion is defined on compact sets by minimizing the same energy over  all smooth functions greater than the characteristic function of the set and then extended to Borel sets as usual.  Therefore, thanks to this approach, we can also define this notion of parabolic capacity of a Borel set with respect to any open subset $U$ of  $Q$ and this will turn out to be very useful in what follows (see for instance Lemma \ref{thanks} below).
\end{remark}

In \cite{dpp} the authors also proved the following decomposition theorem: 

\begin{theorem}\label{cap2}
Let $\mu$ be a bounded measure on $Q$. If $\mu\in M_{0}(Q)$ then there exist $h\in \pw-1p'$, $g\in
L^{p}(0,T;\sob \cap L^{2}(\Omega))$ and $f\in L^{1}(Q)$, such that
\begin{equation}\label{cap3}
\intq \varphi\ d\mu=\int_{0}^{T}\langle h,\varphi \rangle\
dt-\int_{0}^{T}\langle\varphi_{t}, g \rangle\ dt+\intq
f\varphi\ dxdt,
\end{equation}
for any $\varphi\in C_{c}^{\infty}([0,T]\times\Omega)$, where
$\langle\cdot,\cdot\rangle$ denotes the duality between
$V'$ and $V$.
\end{theorem}
\begin{proof}
See \cite{dpp}, Theorem $1.1$.
\end{proof}\medskip

So, if $\mu$ is in $M(Q)$, thanks to a well known decomposition
result (see for instance \cite{fst}), we can split it into a sum
(uniquely determined) of its \emph{absolutely continuous} part
$\mu_{0}$ with respect to $p$-capacity and its \emph{singular}
part $\mu_{s}$,  that is $\mu_{s}$ is concentrated on a set $E$ ( $\mu_{s}=\mu\LL E$) of zero
$p$-capacity; we will say that $\mu_{s}\perp\capp$. Hence, if $\mu\in M(Q)$, by Theorem \ref{cap2}, we have
\begin{equation}
\label{dec} \mu=f-\div (G)+ g_t +\mu^{+}_{s}-\mu_{s}^{-}, \end{equation}
 in the sense
of distributions, for some $f\in L^{1}(Q)$, $G\in (L^{p'}(Q))^{N}$, $g\in\psob$, where $\mu_{s}^{+}$ and $\mu_{s}^{-}$ are respectively the positive and the negative part of  $\mu_{s}$;  note
that the decomposition of the absolutely continuous part of $\mu$
in Theorem \ref{cap2} is not uniquely determined.

Moreover, letting
\[
W_1 =\{u\in\psob,\ u_t\in \pw-1p'\},
\]
then, in our setting, any function of $W_1 $ will admit a $\capp$-quasi continuous representative (that is, it coincides
$\capp$-quasi everywhere with a function which is continuous everywhere but on a set of arbitrarily small capacity,
 see \cite{dpp}).
\setcounter{equation}{0}
\section{General assumptions}\label{2}

 Let $a:(0,T)\times\Omega\times \rn\rightarrow\rn$ be a Carath\'eodory function
(i.e. $a(\cdot,\cdot,\xi)$ is measurable on $Q$,
$\forall\xi\in\rn$, and $a(t,x,\cdot)$ is continuous on $\rn$ for
a.e. $(t,x)\in Q$) such that the following holds:
\begin{equation} \label{a1}
 a(t,x,\xi)\cdot\xi\geq\alpha|\xi|^{p},
\end{equation}
\begin{equation}\label{a2}
|a(t,x,\xi)|\leq\beta[b(t,x)+|\xi|^{p-1}],
\end{equation}
\begin{equation}\label{a3}
(a(t,x,\xi)-a(t,x,\eta))\cdot(\xi-\eta)>0,
\end{equation}
for almost every $(t,x)\in Q$, for all $\xi,\ \eta\in\rn$ with
$\xi\neq\eta$, where  $p>1$, $\alpha,\ \beta$ are positive constants and
$b$ is a nonnegative function in $L^{p'}(Q)$. For every
$u\in \psob$, let us define the differential operator
\[
A(u)=-\div(a(t,x,\nabla u)),
\]
that, thanks to the assumption on $a$, turns out to be a coercive, monotone operator acting from
$ \psob$ into its dual $\pw-1p'$. We shall deal with solutions of the initial boundary value problem
\begin{equation}\label{base}
\begin{cases}
    u_{t}+A(u)=\mu & \text{in}\ (0,T)\times\Omega \\
    u=0 & \text{on}\  (0,T)\times\partial\Omega\\
    u(0)=\uo & \text{in}\ \Omega,
,
\end{cases}
\end{equation}
where $\mu$ is a measure with bounded variation over $Q$, and $\uo\in\luo$.

 For the sake of exposition we will make a further assumption on the range of $p$; we assume
$p>\frac{2N+1}{N+1}$ that is
a standard assumption that gives good compactness results and we
will assume it throughout the paper. Let us observe that, in this setting, the spaces $W$ and $W_1$ turn out to coincide.

\begin{remark}\label{regular}
As we said before, if $\mu\in L^{p'}(Q) $, and $\uo\in L^{2}(\Omega) $, it is well known that problem (\ref{base}) has a unique
solution $u\in W$ in the variational sense, that is
\[
-\into \uo\ \varphi(0)\ dx-\int_{0}^{T}\langle \varphi_{t}, u  \rangle\ dt+\intq
a(t,x,\nabla u)\cdot\nabla \varphi\ dxdt=\int_{0}^{T}\langle f
,\varphi\rangle_{\w-1p',\sob}\ dt,
\]
for all $\varphi\in W $ such that $\varphi(T)=0$ (see
\cite{l}, and \cite{dpp}); notice that $W$ continuously injects in $C([0,T];L^{2}(\Omega))$, and the initial datum is achieved in this sense.
\end{remark}

In \cite{bbggpv} (for more details see also \cite{bgo}) the
concept of entropy solution of the
 elliptic boundary value problem associated to (\ref{base}) was introduced for a 
measure $\mu\in M_{0}(\Omega)$, that is a measure
 with bounded variation over $\Omega$ which does not charge the sets of zero elliptic $p$-capacity (for definition 
and basic properties see for instance \cite{hkm}); an analogous definition
  was given for the parabolic problem in  \cite{p}.
The entropy solution $u$ of the problem (\ref{base}) exists and is
unique as shown in \cite{p} for $\mu\in L^{1}(Q)$,
 result then  improved  for more general measure data: more precisely, for $\mu\in M_0(Q)$ (see for instance \cite{po} and \cite{dpp} in which the result is proved via
  the notion of renormalized solution that turns out to be equivalent to the one of entropy solution with this kind of data as proved in \cite{dpr}).
   Moreover, the solution is such that $|a(t,x,\nabla u)|\in L^{q}(Q)$ for all $\dys q<1+\frac{1}{(N+1)(p-1)}$,
    even if its  gradient  may not belong to any Lebesgue space. Our purpose is to extend all these definitions to general measure data.

Finally, let us state the
following result that will
 be very useful in the sequel; its proof relies on an easy application of \emph{Egorov theorem}.

\begin{proposition}\label{quell}
Let $\rho_\vare$ be a sequence of $\luq$ functions that converges to $\rho$ weakly in $\luq$, and let $\sigma_\vare $ be a sequence of functions in $\liq$ that is bounded in $\liq$ and converges to $\sigma$ almost everywhere on $Q$. Then
$$
\dys\lim_{\vare\to 0}\intq \rho_\vare \,\sigma_\vare\ dxdt=\intq \rho\,\sigma\ dxdt.
$$
\end{proposition} 
 
\setcounter{equation}{0}
\section{Definition of renormalized solution and main result}\label{3}

\ \  Here we give the definition of renormalized solution following the idea of \cite{bp1}. Let
$T_{k}(s)$ be  the \emph{truncation} at levels $\pm k$, that is $T_{k}(s)=\text{max}(-k,\text{min}(k,s))$, for any $k>0$;

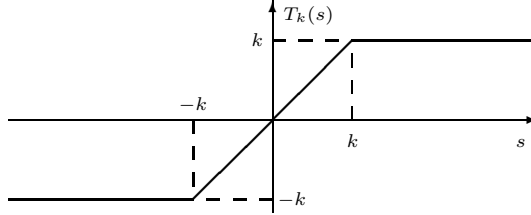
\centerline{
\begin{picture}(200,80)(-100,-35)
\put(-100,0){\vector(1,0){200}}
\put(0,-35){\vector(0,1){80}}
\thicklines
\put(-100,-30){\line(1,0){70}}
\put(-30,-30){\line(1,1){60}}
\put(30,30){\line(1,0){70}}
\thinlines
\multiput(-30,0)(0,-10){3}{\line(0,-1){5}}
\multiput(30,0)(0,10){3}{\line(0,1){5}}
\multiput(0,30)(10,0){3}{\line(1,0){5}}
\multiput(0,-30)(-10,0){3}{\line(-1,0){5}}
\put(-35.5,4){$\scriptstyle -k$}
\put(28,-10){$\scriptstyle k$}
\put(-8,28){$\scriptstyle k$}
\put(2,-32){$\scriptstyle -k$}
\put(92,-10){$\scriptstyle s$}
\put(4,38){$\scriptstyle T_k(s)$}
\end{picture}
}

  To simplify notation, let us also define  in what follows $v=u-g$, where $u$ is the solution, $g$ is the
\emph{time-derivative} part of $\mu_0$, and $\hat{\mu}_0 =\mu-g
_t -\mu_s  =f-\div (G) $; moreover we will write 
\[
\intq  w \ d\hat{\mu}_0 \quad {\rm meaning}\quad\intq fw\ dxdt+\intq G\cdot\nabla w\ dxdt,\]
for every $w\in\psob\cap\liq$.

\begin{definition}\label{1}
Assume \rife{a1}--\rife{a3}, let $\mu\in M(Q)$, and $\uo\in\luo$. A measurable function  $u$ is a
\emph{renormalized solution} of problem \rife{base} if there exists a  decomposition  $(f,G,g)$ of $\mu_0$ such that
   $v\in \plq\cap\pli $ for every $q<p-\frac{N}{N+1}$, $ T_k (v)\in \psob$ for every $k>0$, and for every
$S\in W^{2,\infty}(\re)$  such that $S'$ has 
 compact support on $\re$ and $S(0)=0$, we have
\begin{equation}\label{eq1}
\begin{array}{l}
\dys-\into S(\uo)\varphi(0)\ dx-\intt\langle \varphi_t,S(v)\rangle \ dt + \intq S'(v) a(t,x,\nabla u )\cdot\nabla \varphi \ dxdt\\ \\
+\dys \intq S''(v) a(t,x,\nabla u )\cdot\nabla v \ \varphi\
dxdt=\intq  S'(v) \varphi \ d\hat{\mu}_0,
\end{array}
\end{equation}
for every $\varphi\in\psob\cap L^{\infty}(Q)$, such that $\varphi_t\in\pw-1p'$, and 
$\varphi(T,x)=0$.
 Moreover, for every $\psi\in
C(\overline{Q})$ we have
\begin{equation}\label{rec}
\dys\lim_{n\rightarrow +\infty}\frac{1}{n}\int_{\{n\leq v<
2n\}}\a{u}\cdot\nabla v\ \psi\ dxdt=\intq\psi\ d\mu_{s}^{+}, 
\end{equation}
and
\begin{equation}\label{recm}
\dys\lim_{n\rightarrow +\infty}\frac{1}{n}\int_{\{-2n< v\leq 
-n\}}\a{u}\cdot\nabla v\ \psi\ dxdt=\intq\psi\ d\mu_{s}^{-}.
\end{equation} 
\end{definition}
\begin{remark}
First of all, notice that, thanks to our regularity assumptions and the choice of $S$, all  terms
in \rife{eq1}, \rife{rec}, and \rife{recm} are
well defined; in what follows we will often make a little abuse of notation referring  to $v$ as a \emph{renormalized solution} of problem \rife{base}. Observe that, the  condition $\varphi(T)=0$ is well defined in the sense of $\elle{2}$ thanks to the standard trace result mentioned above.

  Also, observe that \rife{eq1} implies that equation
\begin{equation}\label{eqd}
\begin{array}{l}
S(v)_t -\div(\a{u}S'(v)) + S''(v)\a{u}\cdot\nabla v 
=S'(v)f+G\cdot S''(v)\nabla v -\div(G S'(v))
\end{array}
\end{equation}
is satisfied in the sense of distributions. Since $S(v)_t \in \pw-1p' +\luq$, we can use in \rife{eqd} not only functions in $\C$ but also in $\psob\cap\liq$.
Let us also observe that, since for such $S$ we have  $S(v)=S(T_M (v))\in\psob$ (if $\supp  (S' )\subset [-M,M]$) and  $S(v)_t \in \pw-1p' +\luq$ then $S(v) \in C([0,T];\luo)$ (see \cite{po}). Hence,  one can say that the initial datum is achieved in a weak sense, that is $S(v)(0)=S(\uo)$ in $\luo$ for every $S$ as in Definition \ref{1} (see \cite{dpp} for more details). 

Finally, we want to stress that Definition \ref{1} actually extends all the notions of solutions studied up to now and, in particular, if $\mu\in M_{0}(Q)$ it turns out to coincide with the notion of \emph{entropy solution} as shown in \cite{dpr}; notice that, in this case, entropy and renormalized solutions turn out to be unique. 
\end{remark}

Let us first show the following interesting property of renormalized solutions; throughout the paper  $C$ will indicate any positive constant whose value may change from line to line.
\begin{proposition}\label{natur}
Let $v=u-g$ be a renormalized solution of problem \rife{base}.  Then, for every $k>0$, we have
\begin{equation} \label{natural}
\intq|\nabla T_k (v)|^p\ dxdt\leq \tilde{C}(k+1),
\end{equation}
where $\tilde{C}$ is a positive constant not depending on $k$. 
\end{proposition} 
\begin{proof}
Obviously it is enough to prove \rife{natural} for $k$ large enough.
First of all observe that, thanks to \rife{a1}, \rife{rec} and \rife{recm}, using Young's inequality one can easily show that there exists a positive constant $M$ such that
\begin{equation}\label{natur0}
\frac{1}{n}\int_{\{n\leq |v|<
2n\}}|\nabla u|^p  dxdt\leq M.
\end{equation}
On the other hand, using the definition of $v$, we have
\begin{equation}\label{natur1}
\intq|\nabla T_k (v)|^p\ dxdt\leq C\int_{\{|v|<k\}}|\nabla u|^p dxdt +C.
\end{equation}  
Hence, we have to control the first term on the right hand side of \rife{natur1}; using \rife{natur0}, we have
\[
\begin{array}{l}
   \dys \int_{\{|v|<k\}}|\nabla u|^p dxdt\leq\sum_{n=0}^{[\log_{2}{k}]+1}\int_{\{2^n\leq|v|<2^{n+1}\}}|\nabla u|^p dxdt +\int_{\{0\leq|v|<1\}}|\nabla u|^p dxdt \\\\
   \dys \leq M\sum_{n=0}^{[\log_{2}{k}]+1}2^n\ \  + \ C=M ( 2^{[\log_{2}{k}]+2}-1) +C\leq C(k+1),
\end{array}
\] 
that, together with \rife{natur1} yields \rife{natural}.

\end{proof}\medskip

The main result of this paper is the following one:

\begin{theorem}\label{esi}
Assume \rife{a1}--\rife{a3}, let $\mu\in M(Q)$ and $\uo\in\luo$. Then there exists a renormalized solution of problem \rife{base}.
\end{theorem}

\setcounter{equation}{0}
\section{$\capp$-quasi continuous representative of a renormalized solution}\label{4}

Now we prove some essential properties of renormalized solutions; in particular we shall prove that a renormalized solution (actually the  \emph{regular translation} of it, $v$) is finite cap$_{p}$-quasi everywhere and  admits a cap$_{p}$-quasi continuous representative (which  will be always referred to). To this aim we introduce the following function:  

\begin{equation}\label{hn}
\dys H_n (s)=
\begin{cases}
\quad 1 & \text{if}\ \ |s|\leq n,\\\\
\dys\frac{2n-s}{n}&\text{if}\ \ n<s\leq 2n,\\\\
\dys\frac{2n+s}{n}&\text{if}\ \ -2n<s\leq -n,\\\\
\quad 0&\text{if}\ \ |s|>2n.
\end{cases}
\end{equation}

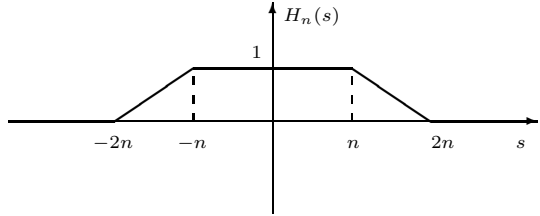
\begin{figure}[!ht]
\centerline{
\begin{picture}(200,80)(-100,-35)
\put(-100,0){\vector(1,0){200}}
 \put(0,-35){\vector(0,1){80}}
\thicklines
 \put(-100,0){\line(1,0){40}}
  \put(-60,0){\line(3,2){30}}
   \put(-30,20){\line(1,0){60}}
 \put(30,20){\line(3,-2){30}}
 \put(60,0){\line(1,0){40}}
\thinlines
 \multiput(30,0)(0,7){3}{\line(0,1){3}}
  \multiput(-30,0)(0,7){3}{\line(0,1){3}}
 \put(28,-10){$\scriptstyle n$}
 \put(-36,-10){$\scriptstyle -n$}
 \put(60,-10){$\scriptstyle 2n$}
 \put(-68,-10){$\scriptstyle -2n$}
 \put(-8,24){$\scriptstyle 1$}
 \put(92,-10){$\scriptstyle s$}
 \put(4,38){$\scriptstyle H_n (s)$}
\end{picture}
}
\end{figure}

 Let us also introduce  another auxiliary function that we will often use in the following; this function can be introduced in terms of  $H_{n}(s)$:
\begin{equation}\label{bn} B_n(s)=1-H_n(s).\end{equation}

\begin{figure}[!ht]
\centerline{
\begin{picture}(200,80)(-100,-35)
\put(-100,0){\vector(1,0){200}}
 \put(0,-35){\vector(0,1){80}}
\thicklines
 \put(-100,20){\line(1,0){40}}
  \put(-60,20){\line(3,-2){30}}
   \put(-30,0){\line(1,0){60}}
 \put(30,0){\line(3,2){30}}
 \put(60,20){\line(1,0){40}}
\thinlines
 \multiput(60,0)(0,7){3}{\line(0,1){3}}
 \multiput(-60,0)(0,7){3}{\line(0,1){3}}
 \multiput(0,20)(7,0){9}{\line(1,0){3}}
  \multiput(-60,20)(7,0){9}{\line(1,0){3}}
 \put(28,-10){$\scriptstyle n$}
 \put(-38,-10){$\scriptstyle -n$}
 \put(60,-10){$\scriptstyle 2n$}
 \put(-70,-10){$\scriptstyle -2n$}
 \put(-8,24){$\scriptstyle 1$}
 \put(92,-10){$\scriptstyle s$}
 \put(4,38){$\scriptstyle B_n (s)$}
\end{picture}
}
\end{figure}

We will also use the following notation: if
$ F$ is a function of one real variable, then $\overline{F}$ will denote its primitive function, that is $\dys\overline{F}(s)=
\int_{0}^{s}F(r)\ dr$; we define the space ${\mathcal S}$ 

\begin{equation}\label{sgrande}
{\mathcal S} = \{ u\in \psob; u_t \in\pw-1p' +  \luq \},
\end{equation}
endowed with its natural norm $\|u\|_{{\mathcal S} }=\|u\|_{\psob} +\|u_t \|_{ \pw-1p' + \luq}$,
 and its subspace $W_2 $ as 
\begin{equation}\label{w2}
W_2 = \{ u\in \psob\cap\liq; u_t \in\pw-1p' +  \luq \},
\end{equation}
endowed with its natural norm $$\|u\|_{W_2}=\|u\|_{\psob}+\|u\|_{\liq}\newline +\|u_t \|_{ \pw-1p' + \luq};$$ for any $p>1$, following the outlines of \cite{dpp} let us also define $\tilde{W}\equiv W_1 \cap L^{\infty}(0,T; L^2 (\Omega))$ and for all $z\in\tilde{W}$, let us denote
$$
[z]_{W} = \|z\|_{\psob}^{p} + \| z_t \|_{\pw-1p'}^{p'} + \|z\|_{\lild}^{2}\,\,\, .
$$ 

In the proof of Lemma $2.17$ of \cite{dpp} the authors show that

\begin{lemma}\label{dpp}
If $u\in W_2 $, then there exists $z\in \tilde{W}$ such that $|u|\leq z$ and 
\begin{equation}\label{ast}
\begin{array}{l}
\dys [z]_{W}\leq C \left( \|u\|^{p}_{\psob}+\|u^{1}_{t}\|_{\pw-1p'}^{p'} \right.\\\\ 
\dys\left.+\|u\|_{\liq}\|u^{2}_{t}\|_{\luq}+ \|u\|^{2}_{\lild} \right),
\end{array}
\end{equation}
where  $u^{1}_{t}\in\pw-1p'$, $u^{2}_{t}\in\luq$ is any decomposition of $u_t $, that is $u_t =u^{1}_{t}+u^{2}_{t}$.
\end{lemma}

\begin{remark}
Observe that $u^{1}_{t}$ and $u^{2}_{t}$ can be chosen such that
$$
\|u^{1}_{t}\|_{\pw-1p'} + \|u^{2}_{t}\|_{\luq}\leq 2 \|u_ t \|_{\pw-1p' + \luq},
$$
and so \rife{ast} easily implies
\begin{equation}\label{astast}
\begin{array}{l}
\dys [z]_{W}\leq C \left( \|u\|^{p}_{\psob}+\|u_{t}\|_{\pw-1p'+\luq}^{p'} \right.\\\\ 
\dys\left.+\|u\|_{\liq}\|u_{t}\|_{\pw-1p'+\luq}+ \|u\|^{2}_{\lild} \right),
\end{array}
\end{equation}
that was a result of Lemma $2.17$ in \cite{dpp}.
For the sake of simplicity let us define
$$
\begin{array}{l}
[u]_{\ast}= \|u\|^{p}_{\psob}+\|u^{1}_{t}\|_{\pw-1p'}^{p'} +\|u\|_{\liq}\|u^{2}_{t}\|_{\luq}+ \|u\|^{2}_{\lild},
\end{array}
$$
and 
$$
\begin{array}{l}
 [u]_{\ast\ast} =\|u\|^{p}_{\psob}+\|u_{t}\|_{\pw-1p'+\luq}^{p'} \\\\ +\|u\|_{\liq}\|u_{t}\|_{\pw-1p'+\luq}+ \|u\|^{2}_{\lild}. 
\end{array}
$$
\end{remark}

Now our aim is to prove the following result: 
\begin{theorem}\label{cqc}
Let $u\in W_2 $; then $u$ admits a unique cap$_p $-quasi continuous representative defined cap$_p$-quasi everywhere.      
\end{theorem}

To prove Theorem \ref{cqc} we need first a \emph{capacitary estimate}, this is the goal of next result:
\begin{lemma}\label{cast}
Let $u\in W_2 $ be a cap$_p$-quasi continuous function. Then, for every $k>0$, 
\begin{equation}\label{cast1}
\capp(\{|u|>k\})\leq \dys\frac{C}{k}\max{\left([u]^{\frac{1}{p}}_{\ast} , [u]^{\frac{1}{p'}}_{\ast}\right)}.
\end{equation}
\end{lemma}
\begin{proof}
We divide the proof in two steps.\newline
 \emph{Step $1$.} Let $u\in\regot$, so  that the set $\{|u|>k\}$ is open and we can estimate its $p$-capacity in terms of the norm of $W$. By Lemma \ref{dpp} there exists $z\in\tilde{W}$ such that $|u|\leq z$ and $\rife{ast}$ holds true, so, recalling that $\tilde{W}\subseteq W $ continuously, $\dys\frac{z}{k}$ is a good function to test $p$-capacity of the set $\{|u|>k\}$ and we can write

\begin{eqnarray*}
\dys\capp(\{|u|>k\})\leq\frac{\|z\|_{W}}{k}\leq \frac{C}{k}\|z\|_{\tilde{W}}
\dys=\frac{C}{k}\left(\|z\|_{\psob}\right. \\\\ \dys \left.+\|z\|_{\lild}+\|z_t \|_{\pw-1p'}\right)\leq
\frac{C}{k}\left([u]^{\frac{1}{p}}_{\ast}+[u]^{\frac{1}{2}}_{\ast}+[u]^{\frac{1}{p'}}_{\ast}\right)\\\\ \dys \leq \frac{C}{k}\max\left([u]^{\frac{1}{p}}_{\ast},[u]^{\frac{1}{p'}}_{\ast}\right),\ \ \ \ \ \ \ \ \ \ \ \ \ \ \ \ \ \ 
\end{eqnarray*}  
as desired.
\emph{Step $2$.} Let $u\in W_2 $ be $\capp$-quasi continuous. Then, for every fixed $\vare >0$,  there exists an open set $A_\vare $ such that $\capp(A_\vare)\leq \vare$ and $u_{\rvert_{Q\backslash A_\vare}}$ is continuous. Hence, the set $\{|u_{\rvert_{Q\backslash A_\vare}}|>k\}\cap (Q\backslash A_\vare )$ is open in $Q\backslash A_\vare $, that is there exists an open set $U\in \mathbb{R}^{N+1}$ such that 
$$
\{|u_{\rvert_{Q\backslash A_\vare}}|>k\}\cap (Q\backslash A_\vare )= U\cap\left(Q\backslash A_\vare\right).
$$ 
Therefore, the set
$$
\{|u|>k\}\cup A_\vare=\{|u_{\rvert_{Q\backslash A_\vare}}|>k\}\cap (Q\backslash A_\vare )\cup A_\vare=(U\cup A_\vare)\cap Q
$$
turns out to be open; let $z$ be the function given in Lemma \ref{dpp},  let $w\in W$ be such that $w\geq \chi_{A_\vare } $ and
$$
\|w\|_{W}\leq \capp(A_\vare ) +\vare\leq 2\vare;
$$
we have that $\dys w+\frac{z}{k}\geq 1$ almost everywhere on $\{|u|>k\}\cup A_\vare $, so that 
\begin{eqnarray*}
\dys\capp (\{|u|>k\}\cup A_\vare)\leq \|w\|_{W}+\frac{\|z\|_{W}}{k}\\\\
\dys\leq \frac{\|z\|_{W}}{k} +2\vare.
\end{eqnarray*}
Finally, using the monotonicity of the  capacity and thanks the arbitrary choice of $\vare$ we can conclude as in Step $1$.    
\end{proof}\medskip

An interesting consequence of these results, whose proof can be performed arguing as in \cite{dmop} and \cite{hkm} for the elliptic case, is the following

\begin{corollary}\label{quella} Let $u\in W_2 $   and $\mu_0\in M_0 (Q)$. Then (the $\capp$ quasi continuous representative of) $u$ is measurable with
respect to $\mu_0 $. If $u$ further is in $L^{\infty}(Q)$, then (the $\capp$-quasi continuous representative of) $u$
belongs to $L^{\infty}(Q,\mu_0 )$, hence to $L^{1}(Q,\mu_0 )$.
\end{corollary}
 
\begin{remark}
Obviously we  also have  that
  \begin{equation}\label{cast2}
\capp(\{|u|>k\})\leq \dys\frac{C}{k}\max{\left([u]^{\frac{1}{p}}_{\ast\ast} , [u]^{\frac{1}{p'}}_{\ast\ast}\right)};
\end{equation}
therefore, thanks to Young's inequality and to the fact that $W_2 $ is continuously embedded in $C([0,T]; L^1 (\Omega))$ (see \cite{po})  we deduce that
\begin{equation}\label{cast3}
\capp(\{|u|>k\})\leq \dys\frac{C}{k}\max{\left(\|u\|^{p}_{ W_2 } , \|u\|^{p'}_{ W_2}\right)}.
\end{equation}
\end{remark}

\begin{proof}[Proof of Theorem \ref{cqc}]
Let us first observe that there are no difficulties in approximating, for instance via convolution,  a function $u\in W_2 $ with smooth functions $u^m \in \regot$  in the norm
 $$\dys \|u^{m}\|_{\psob} +\|u^{m}_{t} \|_{\pw-1p' + \luq} $$ with $\|u^m \|_{\liq}\leq C$ (see \cite{dro} and \cite{dpp}); so, let $u^m $ be a sequence like this for which it is not restrictive to assume that 
  $$
 \sum_{m=1}^{\infty} 2^m \max\left([u^{m+1} - u^m]^{\frac{1}{p}}_{\ast\ast},[u^{m+1} - u^m]^{\frac{1}{p'}}_{\ast\ast}\right)
 $$
 is finite. Now, for every $m$ and $r$, let us define
 $$
 \dys\omega^m =\{|u^{m+1} - u^m |>\frac{1}{2^m}\} \ \ \ \ \text{and} \ \ \ \Omega^{r} =\bigcup_{m\geq r} \omega^m. 
 $$
Applying Lemma \ref{cast} and recalling \rife{cast2} we obtain
$$
\capp (\omega^m )\leq C 2^m \max\left([u^{m+1} - u^m]^{\frac{1}{p}}_{\ast\ast},[u^{m+1} - u^m]^{\frac{1}{p'}}_{\ast\ast}\right),
$$
and so, by subadditivity
$$
\capp (\Omega^r )\leq C\sum_{m\geq r  }2^m \max\left([u^{m+1} - u^m]^{\frac{1}{p}}_{\ast\ast},[u^{m+1} - u^m]^{\frac{1}{p'}}_{\ast\ast}\right), 
$$
that implies 
\begin{equation}\label{azero}
\lim_{r\to\infty}\capp (\Omega^r )=0.
\end{equation}
Moreover, for every $y\notin\Omega^r $ we have that 
$$
|u^{m+1} - u^m|(y)\leq \frac{1}{2^m }
$$
 for any $m\geq r$, and so $u^m$ converges uniformly outside of $\Omega^r $ and pointwise outside of $\dys\bigcap_{r=1}^{\infty}\Omega^r $.  But, for any $l\in \na$, we have
 $$ 
 \capp \left(\dys\bigcap_{r=1}^{\infty}\Omega^r \right)\leq \capp(\Omega^l ),
 $$
 and so, by \rife{azero}, we conclude that  $\dys\capp \left(\dys\bigcap_{r=1}^{\infty}\Omega^r \right)=0$; therefore the limit of $u^m $ is $\capp$-quasi continuous and is defined $\capp$-quasi everywhere.
 
 Let us call $\ut$ this $\capp$-quasi continuous representative of $u$, and let $z$ be another $\capp$-quasi continuous representative of $u$; thanks to Lemma \ref{cast} (in particular using its consequence \rife{cast3}), for any $\vare>0$,   we have
 $$
\capp(\{|\ut-z|>\vare\})\leq \dys\frac{C}{\vare}\max{\left(\|\ut-z \|^{p}_{ W_2 } , \|\ut-z \|^{p'}_{ W_2}\right)}=0
$$
 since $\ut=z$ in $W_2 $. This concludes the proof.
\end{proof}\medskip

Now we want to prove that, if $v=u-g$ is a renormalized  solution then it is finite $\capp$-quasi everywhere and it admits a $\capp$-quasi continuous representative.

\begin{theorem}\label{rcqc}
Let $v=u-g $ be a renormalized solution of problem \rife{base}. Then $v$ admits a $\capp$-quasi continuous representative which is finite $\capp$-quasi everywhere.  
\end{theorem}
\begin{proof}
Let us denote by  $\overline{H}_{n}(s)$ the primitive function of $H_{n}(s)$. Observe that from \rife{natural} we readily have a similar estimate for $\ohnv$, that is 
\begin{equation}\label{esthn}
\intq|\nabla \ohnv|^p\ dxdt\leq C(n+1)
\end{equation}
and so, choosing $\ohnv$  and $\varphi$ in the renormalized formulation \rife{eq1} for $v$,  
we have 
 \begin{equation}\label{mettohn}
\begin{array}{l}
\dys-\intq \varphi_t \  \ohnv\ dxdt \\\\
\dys\qquad+ \intq \hnv a(t,x,\nabla u )\cdot\nabla \varphi \ dxdt\\\\
\dys\quad=\intq  \hnv \varphi \ d\hat{\mu}_0\\\\
\qquad+\dys \frac{1}{n}\int_{\{n< v\leq 2n\}}  a(t,x,\nabla u )\cdot\nabla v \ \varphi\ dxdt\\\\
\dys- \frac{1}{n}\int_{\{-2n< v\leq -n\}}  a(t,x,\nabla u )\cdot\nabla v \ \varphi\ dxdt.
\end{array}
\end{equation}
Hence we deduce that, in the sense of distributions
$$
\dys\frac{d}{dt}\left(\ohnv\right)\in \pw-1p' +L^{1}(Q);
$$
therefore, since $\ohnv \in \psob\cap\liq$, thanks to Theorem \ref{cqc}, $\ohnv$ has a $\capp$-quasi continuous representative  which is finite $\capp$-quasi everywhere. Therefore, to conclude the proof is enough to prove that $v$ is finite $\capp$-quasi everywhere.

Actually, from \rife{mettohn} we have
\begin{equation*}
\begin{array}{l}
\dys\frac{d}{dt}\left(\ohnv\right)= \div(\hnv\a{u})
\dys+ \frac{1}{n}  a(t,x,\nabla u )\cdot\nabla v\chi_{\{n< v\leq 2n\}}\\\\
\qquad\ \ \ \dys -\frac{1}{n}  a(t,x,\nabla u )\cdot\nabla v\chi_{\{-2n< v\leq -n\}}\ +\hnv\mh,
\end{array}
\end{equation*}
and so, from  \rife{esthn}, we easily deduce that there exists a decomposition of $\dys(\ohnv)_t $ such that 
$$
\|(\ohnv)_{t}^{1}\|_{\pw-1p'}^{p'}\leq C n,
$$
and
$$
\|(\ohnv)_{t}^{2}\|_{\luq}\leq C. 
$$

Now, thanks to Theorem \ref{cqc},  we have that $\ohnv$ admits a $\capp$-quasi continuous representative. Since $\{|v|> n\}\equiv \{|\ohnv|>n\}$, we can apply Lemma \ref{cast} to obtain

\begin{equation}\label{manna}
\begin{array}{l}
\dys \capp(\{|v|> n\}) =\capp(\{|\ohnv|>n\})    \leq  \frac{C}{n}\max\left([\ohnv]_{\ast}^{\frac{1}{p}},[\ohnv]_{\ast}^{\frac{1}{p'}}\right);
\end{array}
\end{equation}
hence, using the obvious estimate
$$
\|\ohnv\|_{\lild}^{2}\leq\|\ohnv\|_{\liq} \|\ohnv\|_{\pli}, 
$$
the fact that $v\in\pli$, and \rife{esthn}, we can conclude from \rife{manna} that
$$
\dys \capp(\{|v|> n\})\leq \frac{C}{n}\max(n^{\frac{1}{p}},n^{\frac{1}{p'}}),
$$
so that $v$ is finite $\capp$-quasi everywhere.
\end{proof}\medskip

\setcounter{equation}{0}
\section{Approximating measures; basic estimates and compactness}\label{5}

First of all, we want to emphasize how the renormalized solution does not depend on the decomposition of the regular part of the measure $\mu_{0}$; the proof of this fact essentially relies on the following result proved in \cite{dpp}:
\begin{lemma}\label{ind}
Let $\mu_{0}\in M_{0}(Q)$, and let $(f,- \div (G_1),g_2 )$ and  $ (\tilde{f}, - {\rm div}{(\tilde{G}_1}),\tilde{g}_2)$  be two different decompositions of $\mu$ according to Theorem \ref{cap2}. Then we have $\dys(g_2 - \tilde{g}_2)_t =\tilde{f} -f - \div (\tilde{G}_1) + \div (G_1)$ in distributional sense, $\dys g_2 - \tilde{g}_2\in C([0,T]; L^{1}(\Omega))$ and $\dys(g_2 - \tilde{g}_2)(0)=0$.
\end{lemma}
\begin{proof}
See Lemma $2.29$ in \cite{dpp}, pag. $22$.
\end{proof}\medskip

If $\mu\in M_0 (Q)$, the definition of renormalized solution does not depend on the decomposition of the absolutely continuous part of $\mu$ as shown in  Proposition $3.10$ in \cite{dpp}. The next result tries to  stress the fact that even for general measure data this fact should be true;  we will actually prove that the definition of renormalized solution is stable under \emph{bounded perturbations} of the decomposition of $\mu_0$ (see also Remark \ref{ultimo}).  Here and in the rest of the paper $\omega(\nu,\eta,\vare,n,h,k)$ will  indicate any quantity that vanishes as the parameters go to
their (obvious, if not explicitly stressed) limit point with the same order in which they appear,  that is, for example
\[
\dys\lim_{\nu\rightarrow 0}\limsup_{\delta\rightarrow 0}\limsup_{n\rightarrow +\infty}
\limsup_{\vare\rightarrow 0} |\omega(\vare,n,\delta,\nu)|=0.
\]

\begin{proposition}\label{indi}
Let $u$ be a renormalized solution  of \rife{base}. Then $u$ satisfies Definition \ref{1} for every decomposition $(\tilde{f}, -\div (\tilde{G}),\tilde{g})$ such that $g-\tilde{g}\in \psob\cap\liq$.
\end{proposition}
\begin{proof}[Sketch of the Proof] 
Assume that $u$ satisfies Definition \ref{1} for  $(f, -\div{(G)},g )$ and let $ (\tilde{f}, - \div{(\tilde{G}}),\tilde{g})$ be a different decomposition of $\mu_0 $ such that $g-\tilde{g}$ is bounded. 
Thanks to Lemma \ref{ind} we readily have that $\tilde{v}=u-\tilde{g} \in\pli$; to prove that $T_k (u-\tilde{g})\in \psob$ for every $k>0$ we can reason as in the proof of Proposition $3.10$ in \cite{dpp}, choosing $S(v)=\ohnv$ and  using the fact that thanks to \rife{rec} and \rife{recm} we have
$$
\dys\lim_{n\to\infty}\frac{1}{n}\int_{\{n\leq |u-g |<2n\}}|\nabla u |^{p}\ dx \leq C.
$$  

To prove that the reconstruction properties \rife{rec} and \rife{recm} are satisfied for $\tilde{v}$ we have to be more careful. 

To prove \rife{rec}  we choose $\beta_{h}( \overline{h}_n (v) +g -\tilde{g})\psi$ (where $\beta_h (s) = B_h (s^+)$ and $\psi\in C^1 (\overline{Q})$) and $S' (s)=h_n (s)=H_n (s^+)$ in \rife{eqd}. Writing $\gamma_n= \overline{h}_n (v) +g -\tilde{g}$ and using Lemma \ref{ind}, to obtain
\begin{ceqnarray}
&&
\dys \int_{0}^{T}\langle(\gamma_n)_t , \beta_h (\gamma_n )\psi \rangle\ dt\clabel{ia}{A}\\
&&
\quad+\dys  \intq h_n(v)\a{u}\cdot\nabla\beta_h (\gamma_n)\ \psi\ dxdt\clabel{ib}{B}\\
&&
\qquad\dys -\frac{1}{n}\int_{\{n\leq v< 2n\}}\a{u}\cdot\nabla v \ \ \beta_h (\gamma_n)\ \psi\ dxdt\clabel{ic}{C}\\
&&
\quad \dys+ \intq h_n(v)\a{u}\cdot\nabla\psi \ \beta_h (\gamma_n)\  dxdt\clabel{id}{D}\\
&&
=\dys\intq [(h_n (v) -1)f +\tilde{f}]\beta_h (\gamma_n )\psi\ dxdt\clabel{ie}{E}\\
&&
\quad+\dys\intq [(h_n (v) -1)G +\tilde{G}]\cdot\nabla(\beta_h (\gamma_n )\psi)\ dxdt\clabel{if}{F}\\
&&
\qquad\dys +\intq G\cdot\nabla h_n(v) \ \ \beta_h (\gamma_n)\ \psi\ dxdt\, .\clabel{ig}{G}
\end{ceqnarray}
Now, integrating by parts in \rifer{ia} we have
$$
\rifer{ia}=\dys -\intq \beta_h (\gamma_n )  \psi_t \ dxdt +\into   \beta_h (\gamma_n)(T)\psi(T)\ dx - \into   \beta_h (\overline{h}_n (u_0))\psi(0)\ dx=\omega (n,h)\, ; 
$$
while, thanks to the properties of $\beta_h$ and to the fact $h_n(v)$ strongly converges to $1$ in $L^p ((0,T); W^{1,p}(\Omega))$ (this fact essentially relies on the estimate \rife{natural}) we have that
$$
|\rifer{ie}|+|\rifer{if}|=\omega(n,h)\, . $$
 On the other hand, using H\"{o}lder's inequality and the fact that $0\leq \beta_h (\gamma_n )\leq 1$, we have 
$$  
|\rifer{ig}|\leq C\left(\intq |G|^{p'}\ dxdt\right)^{\frac{1}{p'}}\left(\intq|\nabla h_n (v)|^{p}\ dxdt\right)^{\frac{1}{p}}=\omega(n)\,,
$$
again using the fact that $h_n(v)$ strongly converges to $1$ in $L^p ((0,T); W^{1,p}(\Omega))$.

Moreover, since $g-\tilde{g}$ is bounded, we can truncate $v$ on the set $\{|\gamma_n |\leq 2h\}$, that is  $u=T_M(v) + g$ on  $\{|v|\leq M\}\subseteq \{|\gamma_n |\leq 2h\}$, for a suitable $M>0$ not dependent on $n$. Hence, since $T_{2h}(\gamma_n)$ is weakly compact in $\psob$ (actually arguing as in \cite{dpp}, that is  taking $T_{2h}(\gamma_n)$ as test function in \rife{eqd}, one can show that it is bounded, uniformly with respect to $n$, in $\psob$), we have
$$
\begin{array}{l}
\dys\rifer{ib}=\frac{1}{h}\int_{\{h\leq \gamma_n <2h\}}h_{n}(v)\a{u}\nabla\gamma_n \ \psi\ dxdt\\\\
\dys=
\frac{1}{h}\int_{\{h\leq \tilde{v}<2h\}}\a{u}\nabla \tilde{v}\ \psi\ dxdt +\omega(n),
\end{array}
$$
and, on the other hand, since $h_n(v)$ strongly converges to $1$ in $L^p ((0,T); W^{1,p}(\Omega))$, we have
$$
\begin{array}{l}
\dys-\rifer{ic}=\frac{1}{n}\int_{\{n\leq v< 2n\}\cap\{|\gamma_n| \geq 2h\} }\a{u}\cdot\nabla v \ \psi\ dxdt\\\\
\dys\quad+ \frac{1}{n}\int_{\{n\leq v< 2n\}\cap\{|\gamma_n| < 2h\} }\a{u}\cdot\nabla v \ \ \beta_h (\gamma_n)\ \psi\ dxdt\\\\
\qquad\dys =\frac{1}{n}\int_{\{n\leq v< 2n\}\cap\{|\gamma_n| \geq 2h\} }\a{u}\cdot\nabla v \ \psi\ dxdt+ \omega(n)\\\\
\dys\qquad = \frac{1}{n}\int_{\{n\leq v< 2n\} }\a{u}\cdot\nabla v \ \psi\ dxdt\\\\
\dys- \frac{1}{n}\int_{\{n\leq v< 2n\}\cap\{ |\gamma_n | <2h \} }\a{u}\cdot\nabla v \ \psi\ dxdt +\omega(n)\\\\
\dys\qquad =\intq \psi\ d\msp +\omega(n);
\end{array}
$$
Collecting together all these facts we derive that
$$
\lim_{h\to\infty}\frac{1}{h}\int_{\{h\leq \tilde{v}<2h\}}\a{u}\nabla \tilde{v}\ \psi\ dxdt=\intq \psi\ d\msp. 
$$
Then we conclude by density for every $\psi\in C(\overline{Q})$; the proof of \rife{recm} can be treated analogously.

 Finally the fact that $\tilde{v}$ satisfies equation \rifer{eqd} can be proved as in the proof of Proposition $3.10$ in \cite{dpp}.
\end{proof}
\begin{remark}\label{ultimo}
Let us stress the fact  that in Proposition \ref{indi} we deal with bounded perturbations of the \emph{time derivative} part of $\mu_0$ because of technical reasons; actually the proof of Proposition $3.10$ of \cite{dpp} is given by suitable estimates and with the use of Fatou's  lemma. Unfortunately, as far as the reconstruction properties \rife{rec} and \rife{recm} are concerned, we have to make use of  a more subtle analysis on each term. The requirement of boundedness on $g-\tilde{g}$ arises from this fact.

However, in the linear case,  we can drop this stronger assumption proving the result in its general form. Indeed as we will see later (see the proof of Theorem \ref{tlinear}), in this case a renormalized solution turns out to be a duality solution and so, using the duality formulation for $u$ and Lemma \ref{ind}, we can easily conclude.          
\end{remark}

 Now, let us come back to the existence of a renormalized solution for problem \rife{1}; as we said before, if $\mu\in M(Q)$ we can split it this way:
\begin{equation}\label{app0}
 \mu=\mu_0 +\mu_s=f-\div (G) + g_t+ \mu_{s}^{+} -\mu_{s}^{-}, \end{equation}
 for some $f\in L^{1}(Q)$, $G\in (L^{p'}(Q))^{N}$, $g\in\pw-1p'$, and $\mu_{s}\perp\capp$, that is, $\mu_s$ is concentrated on a set $E\subset Q$ with $\capp(E)=0$. There are many ways to approximate this measure
looking for existence of solutions for problem \rife{base}; we will make the following choice: let
\begin{equation}\label{app}
\mu^{\vare}=f^\vare -\div( G^{\vare} )+g^{\vare}_{t} +\lep -\lem,
\end{equation}
where $f^{\vare}\in \C$ is a sequence of functions
which converges to $f$ weakly in $L^1 (Q)$, $G^{\vare}\in \C$ is a sequence of functions
which converges to $G $ strongly in $(L^{p'} (Q))^N$, $g^{\vare}\in \C$ is a sequence of functions
which converges to $g$ strongly in $\psob$, and $\lep\in \C$ (respectively  $\lem$) is a sequence of nonnegative functions
that converges to $\mu_{s}^{+}$ (respectively $ \mu_{s}^{-}$) in the narrow topology of measures. Moreover let $\uoe\in C^{\infty}_{0}(\Omega)$ be a sequence converging to $\uo$ strongly in $\luo$. 
Notice that this approximation can be easily obtained via a standard convolution argument. We  also assume
\[
\|\mu^{\vare}\|_{\luq}\leq C \|\mu\|_{M(Q)}, \ \ \ \ \|\uoe\|_{\luo}\leq C\|\uo\|_{\luo}\,.
\]

Let us define $u^\vare$ the solution of problem
\begin{equation}\label{basee}
\begin{cases}
    u_{t}^{\vare}+A(u^\vare)=\mu^\vare & \text{in}\ (0,T)\times\Omega \\
u^\vare=0 & \text{on}\  (0,T)\times\partial\Omega,\\
    u^\vare (0)=\uoe & \text{in}\ \Omega,
\end{cases}
\end{equation}
that exists and is unique (see Remark \ref{regular} above), and let $v^\vare=u^\vare- g^\vare $. Approximation
\rife{app} yields standard compactness results (see  \cite{bdgo}, \cite{do}, and \cite{dpp}) that we collect in the following 
\begin{proposition}\label{pro}
Let $u^\vare$ and $v^\vare$ be defined as before. Then
\begin{equation}\label{pro1}
\|u^\vare\|_{L^\infty (0,T; L^1 (\Omega))}\leq C,
\end{equation}
\begin{equation}\label{pro2}
\intq|\nabla T_k (u^\vare)|^p\ dxdt\leq Ck,
\end{equation}
\begin{equation}\label{pro3}
\|v^\vare\|_{L^\infty (0,T; L^1 (\Omega))}\leq C,
\end{equation}
\begin{equation}\label{pro4}
\intq|\nabla T_k (v^\vare)|^p\ dxdt\leq C(k+1).
\end{equation}
Moreover, there exists a measurable function $u$ such that $T_k (u)$ and   $T_k (v)$ belong to $\psob$, $u$ and $v$ belong to $L^\infty (0,T; L^1 (\Omega))$, and, up to  a subsequence, for any $k>0$, and for every $q<p-\frac{N}{N+1}$, we have
\[
\begin{array}{l}
u^\vare\longrightarrow u \ \ \text{a.e. on}\ \ Q \ \ \text{weakly in}\ L^q (0,T;W^{1,q}_{0}(\Omega))\ \ \text{and strongly in} \ L^1 (Q),\\\\
v^\vare\longrightarrow v \ \ \text{a.e. on}\ \ Q\ \ \text{weakly in}\ L^q (0,T;W^{1,q}_{0}(\Omega))\ \ \text{and strongly in} \ L^1 (Q),\\\\
T_k (u^\vare)\rightharpoonup T_k (u)\ \ \ \text{weakly in}\ \ \psob \ \ \text{and a.e. on} \ Q,\\\\
T_k (v^\vare)\rightharpoonup T_k (v)\ \ \ \text{weakly in}\ \ \psob \ \ \text{and a.e. on} \ Q,\\\\
\nabla u^\vare\longrightarrow \nabla u \ \ \text{a.e. on}\ \ Q,\\\\
\nabla v^\vare\longrightarrow \nabla v \ \ \text{a.e. on}\  \ Q.\\\\
\end{array}
\]
\end{proposition}
\begin{remark} 
Let us observe that from Proposition \ref{pro}, thanks to assumption \rife{a2} on $a$ and Vitali's theorem, we  easily deduce that $\a{\ue }$ is strongly compact in $\luq$.      
\end{remark}

\setcounter{equation}{0}
\section{Strong convergence of truncates}\label{6}

\ \ In this section we shall prove the strong convergence of truncates of renormalized solutions of problem \rife{base}; to do
that we will put together the approach used in \cite{dmop} for the elliptic case with the one in \cite{bp1}.

 With the symbol $\tkvn$ we  indicate the Landes time-regularization of the truncate function $T_{k}(v)$; this notion, introduced in \cite{lan}, was fruitfully used in several papers afterwards (see in particular  \cite{do}, \cite{bdgo}, and \cite{bp1}). It is defined as follows:  let $z_\nu$ be a sequence of functions such that
\[
 \begin{array}{l}
 \dys z_\nu\in\sob\cap\lio\,, \ \ \ \ \|z_\nu\|_{\lio}\leq k\,,\\\\
 \dys z_\nu\longrightarrow T_k (\uo)\ \ \ \text{a.e. in}\ \ \Omega\ \ \text{as $\nu$ tends to infinity},\\\\
 \dys\frac{1}{\nu}\|z_\nu\|^{p}_{\sob}\longrightarrow 0 \ \ \text{as $\nu$ tends to infinity}.
 \end{array}
 \]
Then, for fixed $k>0$, and $\nu>0$, we denote by  $\tkvn$  the unique solution of the problem  
 \begin{equation*}
\begin{cases}
   \dys\frac{d T_k (v)}{dt}=\nu(T_k (v)-\tkvn)\ \ \ \text{in the sense of distributions},\\\\
    \dys\tkvn(0)=z_\nu \ \ \ \  \text{in}\ \Omega.
  \end{cases}
\end{equation*}
Therefore, $\tkvn\in\psob\cap\liq$ and $\dys \frac{d T_k (v)}{dt}\in\psob$. It can be proved (see \cite{lan}) that, up to subsequences, as $\nu$ diverges 
\[
\begin{array}{l}
\dys\tkvn\longrightarrow \tkv\ \ \ \ \text{strongly in $\psob$ and a.e. in $Q$},\\\\
\dys\|\tkvn\|_{\liq}\leq k \ \ \ \ \forall\nu>0.
\end{array}
\]

 First of all, let us state a preliminary result about the capacity of compact sets, and then our basic result about approximate capacitary potential.
 \begin{lemma}\label{thanks}
 Let $K$ be a compact subset of  $Q=(0,T)\times\Omega$ such that $\capp(K,Q)=0$, then for every open set $U$ such that $K\subseteq U\subseteq Q$, we have
 $$
  \capp(K,U)=0.
 $$ 
 \end{lemma}
\begin{proof}
Let $\pd\in \regq$ be a sequence approximating the capacity of $K$ in $Q$ (see Remark \ref{rdefi2}), and $\varphi$ be a cut-off function for $K$ in $U$, that is a function in $C^{\infty}_{0}(U)$ such that $\varphi\equiv 1 $ on $K$ and extended to zero on  $Q\backslash U$; therefore we have
$$
\begin{array}{l}
\capp(K,U)\leq \|\pd\varphi\|_{W}= \|\pd\varphi\|_{\psob}+\|(\pd\varphi)_t \|_{\pw-1p'};
\end{array}
$$
easily we have that
$$
 \|\pd\varphi\|_{\psob}\leq C\|\pd\|_{W}.
$$
On the other hand, if $(\pd)_{t} = -\div(F_\delta)$ in the sense of $\pw-1p'$, we have,  reasoning by a density argument,  that, for every $v\in\psob$
$$
\begin{array}{l}
\dys\langle(\pd\varphi)_t, v \rangle_{\pw-1p',\psob}
\dys=\intq F_\delta\cdot\nabla(v \varphi)\ dxdt 
\dys+\intq \pd\varphi_t v\ dxdt.
\end{array}
$$
Thus, using standard trace results (recalling that $p>\frac{2N +1}{N+1}$), H\"{o}lder's inequality and Sobolev embeddings, one can check that
$$
\|(\pd\varphi)_t \|_{\pw-1p'}\leq C\|\pd\|_{W},
$$
that implies the result thanks to the choice of  $\pd$. 
\end{proof}\medskip 
\begin{lemma}\label{acp}
Let $\ms=\msp-\msm \in M(Q)$ where $\msp$ and $\msm$ are  concentrated, respectively, on two disjoint sets $E^+$ and $E^-$  of zero $p$-capacity. Then, for every $\delta>0$, there exist two compact sets $\kdp\subseteq E^+$ and $\kdm\subseteq E^-$ such that

\begin{equation} \label{acp1}
\msp(E^+\backslash \kdp)\leq \delta, \ \ \ \ \msm(E^-\backslash \kdm)\leq \delta,
\end{equation}
and there exist $\pdp, \ \pdm\in  C^{1}_{0}(Q)$, such that
\begin{equation}\label{acp2}
\pdp,\ \pdm\equiv 1\ \text{respectively on}\ \ \kdp, \ \kdm,
\end{equation}
\begin{equation}\label{acp3}
0\leq\pdp,\ \pdm\leq 1,
\end{equation}
\begin{equation}\label{acp4}
\supp(\pdp)\cap\supp(\pdm)\equiv \emptyset.
\end{equation}
Moreover
\begin{equation}\label{acp5}
\|\pdp\|_{{\mathcal S}}\leq \delta,\ \ \ \|\pdm\|_{{\mathcal S}}\leq \delta,
\end{equation}
and, in particular, there exists a decomposition of $(\pdp)_{t}$ and a decomposition  of $(\pdm)_{t}$ such that
\begin{equation}\label{acp6}
\dys\|(\pdp)_{t}^{1}\|_{\pw-1p'}\leq \frac{\delta}{3},\ \ \ \|(\pdp)_{t}^{2}\|_{\luq}\leq \frac{\delta}{3},
\end{equation}
\begin{equation}\label{acp7}
\dys\|(\pdm)_{t}^{1}\|_{\pw-1p'}\leq \frac{\delta}{3},\ \ \ \|(\pdm)_{t}^{2}\|_{\luq}\leq \frac{\delta}{3}.
\end{equation}
Both $\pdp$ and $\pdm$ converge to zero $\ast$-weakly in $\liq$, in $\luq$, and, up to subsequences, almost everywhere as $\delta$ vanishes. 

Moreover, if $\lep$ and $\lem$ are as in \rife{app} we have 
\begin{equation}\label{acp8}
\intq\pdm\ d\lep=\omega(\vare,\delta),\ \ \ \ \intq\pdm\ d\msp\leq\delta ,
\end{equation}
\begin{equation}\label{acp9}
\intq\pdp\ d\lem=\omega(\vare,\delta),\ \ \ \ \intq\pdp\ d\msm\leq\delta ,
\end{equation}
\begin{equation}\label{acp10}
\intq(1-\pdp\pep)\ d\lep=\omega(\vare,\delta,\eta),\ \ \ \ \intq(1-\pdp\pep)\ d\msp\leq \delta +\eta,
\end{equation}
\begin{equation}\label{acp11}
\intq(1-\pdm\pem)\ d\lem=\omega(\vare,\delta,\eta),\ \ \ \ \intq(1-\pdm\pem)\ d\msm\leq \delta +\eta.
\end{equation}
\end{lemma}
\begin{proof}
Let us fix $\delta>0$, so that,  thanks to the regularity of the measure $\mu_s $, there exist two disjoint compact sets $\kdp\subseteq E^+$ and $\kdm\subseteq E^-$, such that  \rife{acp1} is satisfied, and there exist two  open sets $\udp$ and $\udm$, disjoint,  containing respectively $\kdp$ and $\kdm$.
Now, thanks to Lemma \ref{thanks}, since $\capp(\kdp, Q)=0$ (resp. $\capp(\kdm, Q)=0$), we have that  $\capp(\kdp, \udp)=0$ (resp. $\capp(\kdm, \udm)=0$).
Hence, by definition of parabolic $p$-capacity there exists two functions $\fdp\in C_{0}^{\infty}(\udp)$ (resp. $\fdm\in C_{0}^{\infty}(\udm)$) such that, for any $\delta'>0$, we have 
\begin{equation}\label{sus}
\dys\|\fdp\|_{W}\leq \delta',\ \ \ \ (\text{resp.}\ \ \|\fdm\|_{W}\leq  \delta')
\end{equation}
and 
$$
\dys\fdp\geq\chi_{\kdp}\ \ \ \ \ (\text{resp.}\ \fdm\geq\chi_{\kdm}),
$$
where we have extended these functions to zero, respectively,  in $Q\backslash\udp$ and $Q\backslash\udm$; we will choose the value of $\delta'$ in a suitable way later.

Now, let us define
\begin{equation}\label{sweet}
\pdp= \ohp,\ \ \ \ \ \pdm=\ohm,
\end{equation}
where $\overline{H}(s)$ is the primitive of the continuous function  
\begin{equation}
H(s)=
\begin{cases}
\dys\frac{4}{3} & \text{if}\ \ |s|\leq \frac{1}{2},\\\\
\dys \text{affine} &\text{if}\ \ \frac{1}{2}<|s|\leq 1,\\\\
0&\text{if}\ \ |s|>1.
\end{cases}
\end{equation}
It is easy to see that \rife{acp2}, \rife{acp3} and \rife{acp4} are satisfied. Moreover we have 
$$
\begin{array}{l}
\dys\|\pdp\|_{{\mathcal S}}=\|\pdp\|_{\psob}+\|(\pdp)\|_{\pw-1p' + \luq}\\\\
\dys\leq \|\pdp\|_{\psob}+\|(\pdp)^{1}\|_{\pw-1p'}+\|(\pdp)^{2}\|_{\luq},
\end{array}
$$
for every decomposition of $(\pdp)_t $. From now on we deal only with $\pdp$ since the same argument holds for $\pdm$. Let us observe that, in the sense of distributions, we have  $(\pdp)_t =\hp(\fdp)_t $, and so, if  $(\fdp)_t =-\div(F_{\delta}^{+})$ in $\pw-1p'$, we have that, for any $v\in\psob$
$$
\begin{array}{l}
\dys\langle(\pdp)_t, v \rangle_{\pw-1p',\psob}\\\\
\dys=\intq \hp \ F_{\delta}^{+}\cdot\nabla v\ dxdt 
\dys-\frac{8}{3}\int_{\{\frac{1}{2}\leq \fdp\leq 1\}} F_{\delta}^{+}\cdot\nabla\fdp v\ dxdt.
\end{array}
$$ 
Therefore from \rife{sus} we have
$$
\|\pdp\|_{\psob} \leq C\delta',
$$

$$
\|(\pdp)_{t}^{1}\|_{\pw-1p'}\leq C\delta' ,
$$
and, using Young's inequality,
$$
\|(\pdp)_{t}^{2}\|_{\luq}\leq C(\delta'^{p}+\delta'^{p'}).
$$
So, we can actually choose $\delta'$ small enough such that \rife{acp5}, \rife{acp6} and \rife{acp7} are satisfied; moreover, thanks to a Simon type  compactness result (see \cite{do}) we also have that these functions tends to zero in $\luq$ as $\delta$ goes to zero and so, up to subsequences, almost everywhere; due to this fact the  $\liq$  $\ast$-weak convergence  to zero  as $\delta$ vanishes is obvious. 

Now, if $\lep$ is as in the statement we have, for every $\delta>0$,
$$0\leq \intq\pdm\ d\lep=\intq\pdm\ d\msp +\omega(\vare),$$
while recalling \rife{acp1} we have
 $$
 \begin{array}{l}
 \dys0\leq \intq\pdm\ d\msp=\int_{\udm}\pdm\ d\msp\leq \msp(\udm)\leq \msp(Q\backslash \udp)\\\\
 \dys\qquad\leq \msp(Q\backslash \kdp)=\msp(E^{+}\backslash \kdp)\leq\delta.
 \end{array}
 $$
 Therefore \rife{acp8} is proved, and \rife{acp9} can be obtained analogously.
 Now, let $\delta$ and $\eta$ be two nonnegative fixed values: we have
 $$
 0\leq \intq(1-\pdp\pep)\ d\lep= \intq(1-\pdp\pep)\ d\msp+\omega(\vare);
 $$  
 on the other hand, since $1-\pdp\pep$ is in $C(\overline{Q})$, and is identically zero on $\kdp\cap K^{+}_{\eta}$, using again \rife{acp1} we can obtain
 $$
 \begin{array}{l}
 \dys 0\leq \intq(1-\pdp\pep)\ d\msp=\int_{Q\backslash (\kdp\cap K^{+}_{\eta}) }(1-\pdp\pep)\ d\msp\\\\
 \dys \leq\msp(Q\backslash (\kdp\cap K^{+}_{\eta}))\leq \msp(Q\backslash \kdp)+\msp(Q\backslash K^{+}_{\eta})\leq \delta +\eta.
 \end{array}
 $$  
 This proves \rife{acp10} while the proof of \rife{acp11} is analogous.
\end{proof}\medskip

\begin{remark}
This result is enough to our aim; however, one would like to have a stronger result: can one choose $\pdp$ and $\pdm$ as \emph{uniformly bounded} functions in $\C$ vanishing in the $W_1$ norm? This is an open problem which, for instance,  should allow us to prove the reverse implication of the decomposition Theorem \ref{cap2}; that is, if  $\mu\in M(Q)$ admits a decomposition as in     \rife{cap3}, then  $\mu\in M_{0}(Q) $. 
\end{remark}

In what follows we will always refer to  subsequences  of both $\pdp$ and $\pdm$  that satisfy all the convergence results stated in Lemma \ref{acp}.
  
The essential key in the proof of Theorem \ref{esi} is the following

\begin{theorem}\label{str}
Let $ v^\vare$ and $v$ be as before. Then, for every $k>0$
\[
T_{k}(v^\vare)\longrightarrow T_{k}(v)\ \ \ \text{strongly in}\ \ \psob.
\]
\end{theorem}

\begin{proof}
Our aim is to prove the following asymptotic estimate:
\begin{equation}\label{basic}
\limsup_{\vare\rightarrow 0}\intq \a{u^\vare}\cdot\nabla T_k(v^\vare)\ dxdt \leq \intq \a{u}\cdot\nabla T_k(v)\ dxdt.
\end{equation}

 The result will readily follow from \rife{basic} by a quite standard argument. We shall prove it in several steps.\newline
{\it Step $0$}. Near $E$ and far from $E$.\newline
For every $\delta, \eta>0$, let $\pdp$, $\pep$, $\pdm$, and $\pem$ as in Lemma \ref{acp} and let $E^+$ and $E^-$ be the sets where, respectively, $\mu_{s}^{+}$ and $\mu_{s}^{-}$ are concentrated; setting $\fde= \pdep+\pdem$, we can write
\begin{equation}\label{fn}
\begin{array}{l}
\dys\intq \a{u^\vare}\cdot\nabla (\tkve-\tkvn)\hnve\ dxdt\\\\
=\dys\intq \a{u^\vare}\cdot\nabla (\tkve-\tkvn)\hnve\fde\ dxdt \\\\
 +\dys\intq \a{u^\vare}\cdot\nabla (\tkve-\tkvn)\hnve(1-\fde)\ dxdt\,.
\end{array}
\end{equation}
Now, if $n>k$, since $  \a{T_{2n}(\ue)}\cdot\nabla \tkvn$ is weakly compact in $\pelle1$ as $\vare$ goes to zero,  $\hnve$ converges to $\hnv$ $\ast$-weakly in $\pelle{\infty}$,  and almost everywhere on $Q$, thanks to Proposition \ref{quell}, we have
\[
\begin{array}{l}
\dys\limsup_{\vare\rightarrow 0}\dys\intq \a{u^\vare}\cdot\nabla (\tkve-\tkvn)\hnve\fde\ dxdt\\\\\dys=\limsup_{\vare\rightarrow 0}\left[ \intq \a{u^\vare}\nabla \tkve\fde\ dxdt\right]
-\dys \intq \a{u}\cdot\nabla \tkvn \hnv\fde\ dxdt.
\end{array}
\]
So we have
\[
\begin{array}{l}
\dys\limsup_{\vare,\nu}\dys\intq \a{u^\vare}\cdot\nabla (\tkve-\tkvn)\hnve\fde\ dxdt\\\\\dys= \limsup_{\vare\rightarrow 0}\left[ \intq \a{u^\vare}\nabla \tkve\fde\ dxdt\right]
-\dys \intq \a{u}\cdot\nabla \tkv \hnv\fde\ dxdt.
\end{array}
\]
Since $0\leq \hnv\leq 1 $ and  $\fde$ tends to zero $\ast$-weakly in $L^\infty (Q)$ as  $\delta$ goes 
 to zero,
\[
\dys \intq \a{u}\cdot\nabla \tkv \hnv\fde\ dxdt=\omega(\delta).
\]
 Therefore, if we prove that
\begin{equation}\label{15}
\limsup_{\vare,\delta,\eta} \intq \a{u^\vare}\nabla \tkve\fde\ dxdt\leq 0,
\end{equation}
then we can conclude
\begin{equation}\label{f}
\dys\limsup_{\vare, \nu,\delta, n, \eta}\dys\intq \a{u^\vare}\cdot\nabla (\tkve-\tkvn)\hnve\fde\ dxdt\leq 0.
\end{equation}
{\it Step $1$}. Near to $E$.\newline
Let us check \rife{15}. If $\mu^\vare=\hat{\mu}^{\vare}_{0}+\lep -\lem$, then, choosing $(k-\tkve)\hnve\pdep $ as test
function in the weak formulation of $u^\vare $, defining $\dys\Gamma_{n,k}(s)=\int_0^s (k-T_{k}(r))H_{n}(r) \ dr$, and
integrating by parts, we obtain
\begin{equation*}
\begin{array}{l}
\dys-\intq \Gamma_{n,k}(v^{\vare})\ \frac{d}{dt}(\pdep)\ dxdt \\\\
\qquad+  \dys\intq(k-\tkve)\hnve \a{u^\vare}\cdot\nabla(\pdep)  \ dxdt\\\\
\dys\qquad+\intq\a{u^\vare}\cdot\nabla\hnve \ (k-\tkve) \pdep   \ dxdt\\\\
\dys-\intq \a{u^{\vare}}\cdot\nabla\tkve\ \hnve\pdep\ dxdt\\\\
=\quad\dys\intq (k-\tkve)\hnve \pdep\ d\mhe \\\\
\qquad+ \dys \intq (k-\tkve)\hnve \pdep\ d\lep \\\\
 - \dys \intq (k-\tkve)\hnve \pdep\ d\lem; 
 \end{array}
\end{equation*}
so, for $n>k$, we have 
\begin{ceqnarray}
&&
\intq\a{(\tkve+g^{\vare})}\cdot\nabla(\tkve +g^{\vare}) \pdep   \ dxdt \clabel{2a}{A}\\
&&
\qquad+ \dys \intq (k-\tkve)\hnve \pdep\ d\lep\clabel{2b}{B}\\
&&
=\dys-\intq \Gamma_{n,k}(v^{\vare})\ \frac{d}{dt}(\pdep)\ dxdt\clabel{2c}{C}\\
&&
\qquad+\frac{2k}{n}\int_{\{-2n< \ve\leq-n\}}\a{u^\vare}\cdot\nabla\ve \ \pdep   \ dxdt\clabel{2d}{D}\\
&&
\qquad+  \dys\intq(k-\tkve)\hnve \a{u^\vare}\cdot\nabla(\pdep)  \ dxdt \clabel{2e}{E}\\
&&
\quad\dys-\intq (k-\tkve)\hnve \pdep\ d\mhe \clabel{2f}{F}\\
 &&
 \qquad + \dys \intq (k-\tkve)\hnve \pdep\ d\lem \clabel{2g}{G}\\
 &&
 \dys\qquad+\intq\a{(\tkve+g^{\vare})}\cdot\nabla
g^{\vare}\ \pdep\ dxdt .\clabel{2h}{H}
\end{ceqnarray}
Here we have used the fact that 
\begin{equation}\label{comp}
\begin{array}{l}
\dys\intq \a{u^{\vare}}\cdot\nabla\tkve\ \pdep\ dxdt\\\\
=\dys\intq\a{(\tkve+g^{\vare})}\cdot\nabla\tkve \pdep   \ dxdt\\\\
\dys=\intq\a{(\tkve+g^{\vare})}\cdot\nabla(\tkve +g^{\vare}) \pdep   \ dxdt\\\\ \dys
\quad- \intq\a{(\tkve+g^{\vare})}\cdot\nabla
g^{\vare}\ \pdep\ dxdt.
\end{array}
\end{equation}

Let us analyze term by term using in particular Proposition \ref{pro} and Lemma \ref{acp};  due to the fact that $\Gamma_{n,k}(v^{\vare})$ converges to $\Gamma_{n,k}(v)$ weakly in $\psob$, we obtain, observing that  $\Gamma_{n,k}(v)\in \psob\cap\liq$
\[
\begin{array}{l}
\quad-\rifer{2c}=\dys\intq \Gamma_{n,k}(v)\ \frac{d\pdp}{dt}\ \pep\ dxdt\\\\
\dys \qquad+\dys\intq \Gamma_{n,k}(v)\ \frac{d\pep}{dt}\ \pdp\ dxdt +\omega(\vare)=\omega(\vare,\delta);
\end{array}
\]
now, since $(k-\tkve)\hnve$ converges to  $(k-\tkv)\hnv$ $\ast$-weakly in $\liq$, we have 
\[
\begin{array}{l}
\rifer{2e}=\dys\intq(k-\tkv)\hnv \a{T_{2n} (v)+ g}\cdot\nabla(\pdep)  \ dxdt +\omega(\vare)=\omega(\vare,\delta);
\end{array}
\]
moreover, since $\hnve \pdep$ weakly converges to $\hnv \pdep$ in $\psob$ and using again Lemma \ref{acp}, we easily have 
\[
-\rifer{2f}=\dys\intq (k-\tkv)\hnv \pdep\ d\mh +\omega(\vare)=\omega(\vare,\delta);
\]
while, using \rife{acp9} we have
\[
\begin{array}{l}
\dys0\leq\rifer{2g}\leq 2k\intq\pdep\ d\lem= 2k\intq\pdep\ d\msm +\omega(\vare)=\omega(\vare,\delta),
\end{array}
\]
and we readily have that $\rifer{2h}=\omega(\vare,\delta)$.

It remains to control term \rifer{2d}; we want to stress the fact that the use of the double cut-off function $\pdep$ has been introduced essentially to control this term.  Suppose we proved that $\rifer{2d}=\omega(\vare,\delta,n,\eta)$ and let us conclude the proof of \rife{15}; actually, collecting all we shown above, we have
\[
\begin{array}{l}
\rifer{2a}+\rifer{2b}=\omega(\vare,\delta,n,\eta),
\end{array}
\]
and, observing that both \rifer{2a} and \rifer{2b} are nonnegative, we can conclude that
\begin{equation}\label{ninso}
\intq\a{(\tkve+ g^{\vare} )}\nabla(\tkve +g^{\vare}) \pdep   \ dxdt=\omega(\vare,\delta,\eta),
\end{equation}
and
\begin{equation}\label{lambdae}
\dys \intq (k-\tkve)\hnve \pdep\ d\lep=\omega(\vare,\delta,n,\eta).
\end{equation}

On the other hand, reasoning as before with $(k+\tkve)\hnve\pdem$ as test function we can obtain
\begin{equation}\label{ninsom}
\intq\a{(\tkve+ g^{\vare} )}\nabla(\tkve +g^{\vare}) \pdem   \ dxdt=\omega(\vare,\delta,\eta),
\end{equation}
and
\begin{equation}\label{lambdaem}
\dys \intq (k+\tkve)\hnve \pdem\ d\lem=\omega(\vare,\delta,n,\eta).
\end{equation}

Now, \rife{ninso} and \rife{ninsom} together with \rife{comp} (that obviously holds true with $\pdem$ in place of $\pdep$) yield  \rife{15}, 
 while \rife{lambdae} and \rife{lambdaem} both show  an interesting property of
 approximating 
renormalized solutions; they suggest that, in some sense, $v^\vare$ (and so the
solution $u^\vare$) tends to be, respectively, large (larger than any $k>0$) on the set where the singular
 measure $\mu_{s}^{+}$ is concentrated, and small (smaller than any $k<0$) on the set where the singular
 measure $\mu_{s}^{-}$ is concentrated. 
 
 So, to conclude let us check that  $\rifer{2d}=\omega(\vare,\delta,n,\eta)$ (and the analogous property for $\pdem$). First of all, since $0\leq\pdp\leq 1$, we have that
 \[
 \begin{array}{l}
\rifer{2d} =\dys\frac{2k}{n}\int_{\{-2n< \ve\leq-n\} }\a{(T_{2n}(\ve)+ g^\vare) }\cdot\nabla (T_{2n}(\ve)+g^{\vare}) \ \pdep   \ dxdt\\\\
\dys\quad-\frac{2k}{n}\int_{\{-2n< \ve\leq-n\}}\a{(T_{2n}(\ve)+g^\vare)}\cdot\nabla g^\vare \ \pdep   \ dxdt\\\\ \leq \dys\frac{2k}{n}\int_{\{-2n< \ve\leq-n\} }\a{(T_{2n}(\ve)+ g^\vare) }\cdot\nabla (T_{2n}(\ve)+g^{\vare}) \ \pep   \ dxdt +\omega(\vare,\delta)\\\\
\dys\qquad=\frac{2k}{n}\int_{\{-2n< \ve\leq-n\}}\a{u^\vare}\cdot\nabla\ve \ \pep   \ dxdt+\omega(\vare,\delta,n),
 \end{array}
 \]
 where to get last equality we used \rife{a2}, H\"{o}lder's inequality and the estimate on the truncates of $\ue$ given by Proposition \ref{pro};
 therefore, we have just to prove that
 \[
 \frac{1}{n}\int_{\{-2n< \ve\leq-n\}}\a{u^\vare}\cdot\nabla\ve \ \pep  =\omega(\vare,n,\eta). 
 \]
 
  To emphasize this interesting property that, at first glance, may appear in contrast with the reconstruction property \rife{recm}, 
 we will prove it in the following 
 \begin{lemma}\label{dazero}
 Let  $\ue$ be a solution of problem \rife{basee} and $\pep$, $\pem$ as in Lemma \ref{acp}. Then 
 \begin{equation}\label{dazerop}
 \frac{1}{n}\int_{\{-2n< \ve\leq-n\}}\a{u^\vare}\cdot\nabla\ve \ \pep   \ dxdt=\omega(\vare,n,\eta),
 \end{equation}
 and
  \begin{equation}\label{dazerom}
 \frac{1}{n}\int_{\{n\leq \ve<2n\}}\a{u^\vare}\cdot\nabla\ve \ \pem   \ dxdt=\omega(\vare,n,\eta).
 \end{equation}
 \end{lemma}
 \begin{proof}
Let us prove \rife{dazerom}; if $\beta_n (s)=B_n (s^+ )$, we can choose $\benve \pem$ as test function for problem \rife{basee}, and rearranging conveniently all terms, we have
\begin{ceqnarray}
&&
\dys\frac{1}{n}\int_{\{n\leq\ve<2n\}}\a{(T_{2n}(\ve)+g^{\vare})}\cdot\nabla(T_{2n}(\ve) +g^{\vare}) \pem   \ dxdt \clabel{0a}{A}\\
&&
\qquad+ \dys \intq \benve \pem\ d\lem\clabel{0b}{B}\\
&&
=\dys\intq \overline{\beta}_n (\ve)\frac{d \pem}{dt}\ dxdt\clabel{0c}{C}\\
&&
\quad-\intq \a{\ue}\cdot\nabla\pem\benve\ dxdt \dys\clabel{0d}{D}\\
&&
 \qquad+\dys \intq \benve\pem\ d\mhe\clabel{0e}{E}\\
&&
 \qquad\dys+\intq \benve\pem\ d\lep\clabel{0f}{F}\\
 &&
 \dys\qquad+\frac{1}{n}\int_{\{n\leq\ve<2n\}}\a{(T_{2n}(\ve)+g^{\vare})}\cdot\nabla
g^{\vare}\ \pem\ dxdt.\clabel{0g}{G}
\end{ceqnarray}
  Observing that both \rifer{0a} and \rifer{0b} are nonnegative, let us analyze the right hand side term by term; thanks to Proposition \ref{pro} and to the fact that $\beta_{n} (v) $ converges to $0$ a.e. on $Q$ and $\ast$-weakly in $\liq$, we have    
   \[
   \rifer{0d}=\omega(\vare,n),
   \]
   while, since $\overline{\beta}_{n}(\ve)$ converges to  $\overline{\beta}_{n}(v)$ as $\vare$ goes to zero, and $\overline{\beta}_{n}(v)$ tends to $0$  in $\luq$ as $n$ diverges, again thanks to Proposition \ref{pro} we easily obtain
    \[
   \rifer{0c}=\omega(\vare,n);
   \]
  moreover, again thanks to Proposition \ref{pro} and to the definition of $\beta_n$, in particular using the fact that $\beta_n (v)$ strongly converges to $0$ in $\psob$ (again this fact  is an easy consequence of the estimate on the  truncates of $\ue$ in Proposition \ref{pro}) and $\ast$-weakly in $\liq$,  we have that $\rifer{0e}=\omega(\vare,n)$ and, using \rife{a2} and H\"{o}lder's inequality as before, we have
  \[
 \rife{0g}=\omega(\vare,n);
   \] 
   finally, thanks to \rife{acp8},
   \[
   \rifer{0f}\leq\intq\pem\ d\lep=\omega(\vare,\eta).
   \]
   Putting together all these facts we obtain \rife{dazerom}, while \rife{dazerop} can be proved in an analogous way choosing $B_n (s^- )$ and $\pep$ as test functions in \rife{basee}. \end{proof}\medskip
   \begin{remark}  \label{remesi}
   Notice that the result of Lemma \ref{dazero} turns out to hold true even for more general functions $\pep$ and $\pem$ in $W^{1,\infty}(Q)$ which satisfy
   \[
   0\leq\pep\leq 1 \ \ \ \ \ \ \ \ \ \ 0\leq\pem\leq 1, 
   \] 
   and
   \[
   0\leq\intq\pep\ d\msm\leq \eta \ \ \ \ \ \ \ \ \ \  0\leq\intq\pem\ d\msp\leq \eta,
   \]
   since we can reproduce the same calculations as in the proof of Lemma \ref{dazero}    using  the fact that  the reminder terms of the integration by parts easily vanish  as first $\vare$ goes to  zero and then $n$ diverges. We will use this fact later. 
   \end{remark}
\noindent{\it Step $2$}.  Far from $E$.\newline
We first prove a result that will be essential to deal with the second term in the right hand side of \rife{fn}:
\begin{lemma}\label{fettine}
Let $h,k>0$, and $\ue$ and $\fde$  be as before, then 
\begin{equation}\label{fetteq}
\dys\int_{\{h\leq |\ve|<h+k\}}|\nabla \ue|^p (1-\fde)=\omega(\vare,h,\delta,\eta)
\end{equation}
\end{lemma}
\begin{proof}
Let $\psi(s)=T_{k}(s-T_h (s) )$ and let us multiply the formulation of $\ue$ by the test function $\psi(\ve)(1-\fde)$; integrating, if $\Theta_{k,h}(s)=\dys\int_{0}^{s}\psi(\sigma)\ d\sigma$, we have
\[
\begin{array}{l}
 \dys\intq \Theta_{k,h}(\ve)_t (1-\fde)\ dxdt  \\\\
\dys\qquad+ \intq \a{u^\vare}\cdot\nabla T_k (\ve -T_{h}(\ve))\ (1-\fde)\ dxdt\\\\
\dys\quad- \intq \a{u^\vare}\cdot\nabla \fde\ T_k (\ve -T_{h}(\ve))\ dxdt\\\\
=\dys\intq f^\vare\  T_k (\ve -T_{h}(\ve)) (1-\fde)  \ dxdt\\\\
\dys\qquad+\intq G^\vare\cdot\nabla(T_k (\ve -T_{h}(\ve)) (1-\fde))\ dxdt \\\\
 \dys\qquad+\intq T_k (\ve -T_{h}(\ve)) (1-\fde) d\lep\\\\
\quad-\dys\intq T_k (\ve -T_{h}(\ve)) (1-\fde) d\lem.
\end{array}
\]
Now using the fact that
\begin{equation}\label{yo}
\begin{array}{l}
 \dys\intq \a{u^\vare}\cdot\nabla T_k (\ve -T_{h}(\ve))\ (1-\fde)\ dxdt\\\\
 = \dys \int_{\{h\leq |\ve|<h+k\}} \a{u^\vare}\cdot\nabla \ue\ (1-\fde)\ dxdt \\\\
 \quad-\dys\int_{\{h\leq |\ve|<h+k\}} \a{u^\vare}\cdot\nabla g^\vare\ (1-\fde)\ dxdt,
\end{array}
\end{equation}
and, using Young's inequality we obtain
\[
\begin{array}{l}
\dys\left|\intq G^\vare\cdot\nabla T_k (\ve -T_{h}(\ve)) \  (1-\fde)\ dxdt \right|
\leq C_1 \int_{\{h\leq |\ve|<h+k\}} |G^\vare|^{p'}(1-\fde)\ dxdt\\\\
\dys + C_2 \int_{\{h\leq |\ve|<h+k\}} |\nabla \ue|^p \ (1-\fde)\ dxdt \dys+\left|\int_{\{h\leq |\ve|<h+k\}}G^\vare\cdot\nabla g^\vare (1-\fde)\ dxdt\right|;
\end{array}
\]
where, when we use Young's inequality, we can choose $C_2$ small as we want (for instance $\dys C_2<\frac{\alpha}{3}$); in the same way we can deal with the second term on the right hand side of \rife{yo}  after we used assumption \rife{a2} on $a$; therefore, using  assumption \rife{a1} on $a$  in the first term of the right hand side of \rife{yo}, noticing that $\Theta_{k,h}(s)$ is nonnegative for any $s\in\re$ and integrating by parts, we obtain
\begin{ceqnarray}
&&
\dys\quad \intq  \Theta_{k,h}(\ve)\frac{d\fde}{dt}\ dxdt \clabel{fa}{A}\\
&&
\qquad\dys+ \int_{\{h\leq |\ve|<h+k\}}|\nabla \ue|^p\ (1-\fde)\ dxdt\clabel{fb}{B}\\
&&
\leq\dys C\int_{\{h\leq |\ve|<h+k\}}|G^\vare|^{p'}\ (1-\fde)\ dxdt\clabel{fc}{C}\\
&&
\qquad +\dys C\int_{\{h\leq |\ve|<h+k\}}|\nabla g^\vare|^p \  (1-\fde)\ dxdt\clabel{fd}{D}\\
&&
\qquad +\dys \int_{\{|\ve|\geq h\}}|f^\vare | (1-\fde) \ dxdt\clabel{fe}{E}\\
&&
\dys\qquad+ C\int_{\{h\leq |\ve|<h+k\}}|b(t,x)|^{p'}\ (1-\fde)\ dxdt \clabel{agg}{F}\\
&&
 \qquad\dys+\intq T_k (\ve -T_{h}(\ve)) (1-\fde) d\lep\clabel{ff}{G}\\
 &&
\quad-\dys\intq T_k (\ve -T_{h}(\ve)) (1-\fde) d\lem\clabel{fg}{H}\\
&&
\dys\qquad+ \into\Theta_{k,h}(\uoe)\ dx.\clabel{fi}{I}
\end{ceqnarray}
First of all, thanks to the definition of   $\Theta_{k,h}(s)$, the strong compactness in $\pelle1$ of both $\ve$ and $\uoe$, using  Vitali's theorem we readily have
\[
\begin{array}{l}
|\rifer{fa}|+|\rifer{fi}|=\omega (\vare,h),
\end{array}
\]
while thanks to the equi-integrability property 
\[
\rifer{fc}+\rifer{fd}+\rifer{fe}+\rifer{agg}=\omega(\vare,h);
\]
finally, thanks to \rife{acp8} and \rife{acp10} we have
\[
|\rifer{ff}|\leq k\left|\intq (1-\pdep)\ d\lep -\intq \pdem\ d\lep \right|=\omega(\vare,\delta,\eta);
\]
analogously using \rife{acp9} and \rife{acp11} one has $|\rifer{fg}|=\omega(\vare,\delta,\eta)$; collecting together all these facts we obtain \rife{fetteq}.
\end{proof}\medskip

Now, let us analyze the second term in the right hand side of \rife{fn}, that is, in some sense, far from
the set where the singular measure is concentrated. 

We can write, for $n>k$, 
\begin{equation}\label{strhs}
\begin{array}{l}
\dys\intq \a{u^\vare}\cdot\nabla (\tkve-\tkvn)\hnve(1-\fde)\ dxdt\\\\
=\dys \int_{\{|\ve|\leq k \}} \a{u^\vare}\cdot\nabla (\ve-\tkvn)(1-\fde)\ dxdt\\\\
\quad-\dys \int_{\{|\ve|> k\}} \a{u^\vare}\cdot\nabla \tkvn\hnve(1-\fde)\ dxdt;
\end{array}
\end{equation}
First of all, thanks to Proposition \ref{quell} and to Proposition \ref{pro}, and since, by its definition,  $|\tkvn|\leq k $ a.e. on $Q$, we have
\begin{equation}\label{facile}
\begin{array}{l}
\dys \int_{\{|\ve|> k\}} \a{u^\vare}\cdot\nabla \tkvn\hnve(1-\fde)\ dxdt\\\\
\dys \quad= \int_{\{|\ve|> k\}} \a{(T_{2n} (\ve) +g^\vare )}\cdot\nabla \tkvn\hnve(1-\fde)\ dxdt\\\\
\dys\quad=  \int_{\{|\ve|> k\}} \a{(T_{2n} (v) +g)}\cdot\nabla \tkvn\hnv (1-\fde)\ dxdt +\omega(\vare)\\\\
\dys\quad=\omega(\vare,\nu).
\end{array}
\end{equation}

To deal with the first term on the right hand side of \rife{strhs} we adapt a method introduced, for the parabolic case, in \cite{po}; for $h>2k$ let us define
\[
w^\vare=T_{2k}(\ve-T_h (\ve)+ \tkve -\tkvn);
\] 
notice that $\nabla w^\vare =0$ if $|\ve|>h+4k$, thus the estimate on $\tkve$ of Proposition \ref{pro} implies that $w^\vare$ is bounded in $\psob$; therefore we easily have that 
\[
w^\vare\longrightarrow  T_{2k}(v-T_h (v)+ \tkv -\tkvn) \ \ \ \text{weakly in $\psob$ and a.e. on $Q$.} 
\]
Hence, let us multiply by $w^\vare (1-\fde)$ the equation solved by $\ue$ and  integrate to obtain
\begin{ceqnarray}
&&
\dys\int_{0}^{T}\langle\ve_{t}, w^\vare (1-\fde)\rangle\ dt \clabel{3a}{A}\\
&&
\qquad\dys+ \intq \a{u^\vare}\cdot\nabla w^\vare (1-\fde)\ dxdt\clabel{3b}{B}\\
&&
\quad\dys- \intq \a{u^\vare}\cdot\nabla \fde\ w^\vare\ dxdt\clabel{3c}{C}\\
&&
=\dys\intq f^\vare\  w^\vare (1-\fde)  \ dxdt \clabel{3d}{D}\\
&&
 \qquad\dys+\intq G^\vare\cdot\nabla(w^\vare (1-\fde))\ dxdt \clabel{3e}{E}\\
&&
 \qquad\dys+\intq w^\vare (1-\fde) d\lep\clabel{3f}{F}\\
 &&
\quad-\dys\intq w^\vare (1-\fde) d\lem.\clabel{3g}{G}
\end{ceqnarray}
Let us analyze term by term the above identity; first of all, thanks to the properties of $w^\vare$ and to  Lebesgue's dominated convergence theorem we  have that $\rifer{3d}= \omega(\vare,\nu,h)$; while, on the other hand, we  have
\[
\dys\rifer{3e}=\int_{\{h\leq v<h+2k \}} G\cdot \nabla v\ (1-\fde)\ dxdt + \omega(\vare,\nu,h),
\]
 and using Young's inequality and  Lemma \ref{fettine},  we have that 
 \[
\dys \int_{\{h\leq v<h+2k \}} G\cdot \nabla v\ (1-\fde)\ dxdt=\omega(h,\delta,\eta).
 \]
 Now, reasoning as in the proof of Lemma \ref{fettine}, thanks to \rife{acp8}--\rife{acp11}, using the fact that $|w^\vare|\leq 2k $, we have that both  $\rifer{3f}=\omega(\vare,\delta,\eta)$ and $\rifer{3g}=\omega(\vare,\delta,\eta)$, while thanks to Proposition \ref{pro} and to the definition of $w^\vare$ we have $
 \rifer{3c}=\omega(\vare,\nu,h)$.
 
Let us now analyze term \rife{3b}; if we define $M=h+4k$ we have
\[
\begin{array}{l}
\dys \intq \a{u^\vare}\cdot\nabla w^\vare (1-\fde)\ dxdt= \intq \a{u^\vare\chi_{\{|\ve|\leq M\}}}\cdot\nabla w^\vare (1-\fde)\ dxdt.
\end{array}
\]
Now, if $E_\vare=\{|\ve-T_h (\ve)+ \tkve -\tkvn|\leq 2k\}$ and $h\geq2k$ we can split it as
\begin{equation}\label{332}
\begin{array}{l}
\dys \intq \a{u^\vare}\cdot\nabla w^\vare (1-\fde)\ dxdt\\\\
\dys= \intq \a{u^\vare\chi_{\{|\ve|\leq k\}}}\cdot\nabla (\ve-\tkvn)(1-\fde)\ dxdt\\\\
\dys\qquad+ \int_{\{|\ve|>k\}} \a{u^\vare\chi_{\{|\ve|\leq M\}}}\cdot\nabla (\ve-T_h (\ve))\  (1-\fde)\chi_{E_\vare}\ dxdt\\\\
\dys\quad-\int_{\{|\ve|>k\}} \a{u^\vare\chi_{\{|\ve|\leq M\}}}\cdot\nabla \tkvn \  (1-\fde)\chi_{E_\vare}\ dxdt.
\end{array}
\end{equation}
Let us analyze the second term in the right hand side of  \rife{332};  
since $\ve -T_h (\ve)=0$ if $|\ve|\leq h$, we have
\[
\begin{array}{l}
\dys\left| \int_{\{|\ve|>k\}} \a{u^\vare\chi_{\{|\ve|\leq M\}}}\cdot\nabla (\ve-T_h (\ve))\  (1-\fde)\chi_{E_\vare}\ dxdt\right|\\\\
\dys\quad\leq \int_{\{h\leq|\ve|<h+4k\}}|\a{\ue}||\nabla \ve|\ dxdt,
\end{array}
\]
and using assumption \rife{a2} on $a$ and Young's inequality we get: 
\[
\begin{array}{l}
 \dys\int_{\{h\leq|\ve|<h+4k\}}|\a{\ue}||\nabla \ve|\  (1-\fde)\ dxdt\\\\
 \dys\leq C\int_{\{h\leq|\ve|<h+4k\}}|\nabla\ue|^p \  (1-\fde)\ dxdt\\\\
 \dys \quad+C\int_{\{h\leq|\ve|<h+4k\}}|\nabla g^\vare|^p \  (1-\fde)\ dxdt
 \\\\ \quad\dys+C\int_{\{h\leq|\ve|<h+4k\}}|b(t,x)|^{p'}\  (1-\fde)\ dxdt.
\end{array}
\]
Thus, using equi-integrability and Lemma \ref{fettine} we obtain
\begin{equation}\label{332b}
 \int_{\{|\ve|>k\}} \a{u^\vare\chi_{\{|\ve|\leq M\}}}\cdot\nabla (\ve-T_h (\ve))\  (1-\fde)\chi_{E_\vare}\ dxdt=\omega(\vare,h,\delta,\eta).
\end{equation}
Let us now analyze the third term in the right hand side of \rife{332}; since, thanks to Proposition \ref{pro}, we have 
\[
\int_{\{|\ve|>k\}} \a{u^\vare\chi_{\{|\ve|\leq M\}}}\cdot\nabla \tkv \  (1-\fde)\chi_{E_\vare}\ dxdt=\omega(\vare),
\]
then
\begin{equation}\label{333}
\begin{array}{c}
\dys\int_{\{|\ve|>k\}} \a{u^\vare\chi_{\{|\ve|\leq M\}}}\cdot\nabla \tkvn \  (1-\fde)\chi_{E_\vare}dxdt\\\\
\dys =\int_{\{|\ve|>k\}} \a{u^\vare\chi_{\{|\ve|\leq M\}}}\cdot\nabla (\tkvn-\tkv) \  (1-\fde)\chi_{E_\vare}dxdt+\omega(\vare),
\end{array}
\end{equation}
so, thank to  the fact that $\tkvn$ strongly converges to $\tkv$ in $\psob$ and using again Proposition \ref{pro} we have
\[
\int_{\{|\ve|>k\}} \a{u^\vare\chi_{\{|\ve|\leq M\}}}\cdot\nabla (\tkvn-\tkv) \  (1-\fde)\chi_{E_\vare}\ dxdt=\omega(\vare,\nu),
\]
that together with \rife{333} yields
\begin{equation}\label{332c}
\int_{\{|\ve|>k\}} \a{u^\vare\chi_{\{|\ve|\leq M\}}}\cdot\nabla \tkvn \  (1-\fde)\chi_{E_\vare}\ dxdt=\omega(\vare,\nu).
\end{equation}
So, putting together \rife{332}, \rife{332b} and \rife{332c} we get
\[
\rifer{3b}=\intq \a{u^\vare\chi_{\{|\ve|\leq k\}}}\cdot\nabla (\ve-\tkvn)(1-\fde)\ dxdt+\omega(\vare,\nu,h,\delta,\eta),
\]
 and then, gathering together all the above results
\begin{equation}
\begin{array}{l}
\dys\int_{0}^{T}\langle\ve_{t}, w^\vare (1-\fde)\rangle\ dt \\\\
\dys\quad+\intq \a{u^\vare\chi_{\{|\ve|\leq k\}}}\cdot\nabla (\ve-\tkvn)(1-\fde)\ dxdt\\\\
=\omega(\vare,\nu,h,\delta,\eta).
\end{array}
\end{equation}
If we prove that
\begin{equation}\label{farfrom}
\int_{0}^{T}\langle\ve_{t}, w^\vare (1-\fde)\rangle\ dt \geq\omega(\vare,\nu,h),
\end{equation}
then we obtain our estimate \emph{far from $E$}: 
\begin{equation}\label{g}
\dys\limsup_{\vare,\nu,\delta,n,\eta}\dys\intq \a{u^\vare}\cdot\nabla (\tkve-\tkvn)\hnve(1-\fde)\ dxdt\leq 0.
\end{equation}

So, let us prove \rife{farfrom}. Observing that, thanks to the fact that $|\tkvn|\leq k$, we can write (recalling that $h>k>0$)
\[
w^\vare= T_{h+k}(\ve-\tkvn)-T_{h-k}(\ve-\tkve);
\]
we have, 
\[
\begin{array}{l}
\dys\int_{0}^{T}\langle\ve_{t}, w^\vare (1-\fde)\rangle\ dt \\\\
\dys=\int_{0}^{T}\langle(\tkvn)_{t},T_{h+k}(\ve -\tkvn) (1-\fde)\rangle\ dt\\\\
\dys\qquad+\intq S_{h+k}(\ve-\tkvn)_t \  (1-\fde)\ dxdt\\\\
\dys\quad-\intq G_{h-k}(\ve)_t \  (1-\fde)\ dxdt,
\end{array}
\]
 where
 \[
 S_{h+k}(s)=\int_{0}^{s}T_{h+k}(\sigma)\ d\sigma,
 \]
 and
 \[
 G_{h-k}(s)=\int_{0}^{s}T_{h-k}(\sigma-T_{k}(\sigma))\ d\sigma.
 \]
 First of all, thanks to the definition of $\tkvn$ we have
\[
\begin{array}{l}
\dys\int_{0}^{T}\langle(\tkvn)_{t},T_{h+k}(\ve -\tkvn) (1-\fde)\rangle\ dt\\\\
\dys=\nu\ \intq (\tkv-\tkvn)T_{h+k}(v -\tkvn) (1-\fde)\ dxdt +\omega(\vare)\\\\
\dys= \nu\ \int_{\{|v|\leq k\}}(v-\tkvn)T_{h+k}(v -\tkvn) (1-\fde)\ dxdt\\\\ 
\qquad \dys+\ \nu\  \int_{\{v>k\}}(k-\tkvn)T_{h+k}(v -\tkvn) (1-\fde)\ dxdt \\\\
\dys\qquad +\ \nu\ \int_{\{v<-k\}}(-k-\tkvn)T_{h+k}(v -\tkvn) (1-\fde)\ dxdt +\omega(\vare);
\end{array}
\]
and the three  terms in the right hand side are all nonnegative, so we can drop  them  to obtain
\begin{equation}\label{nu}
\int_{0}^{T}\langle(\tkvn)_{t},T_{h+k}(\ve -\tkvn) (1-\fde)\rangle\ dt\geq \omega(\vare),
\end{equation}
while integrating by parts we have
\[
\begin{array}{l}
\dys\intq S_{h+k}(\ve-\tkvn)_t \  (1-\fde)\ dxdt
\dys-\intq G_{h-k}(\ve)_t \  (1-\fde)\ dxdt\\\\
=\dys \intq S_{h+k}(\ve-\tkvn) \  \frac{d\fde}{dt}\ dxdt 
\dys-\intq G_{h-k}(\ve) \ \frac{d\fde}{dt} \ dxdt\\\\
\dys\dys\qquad+ \into S_{h+k}(\ve-\tkvn)(T) \ dx
\dys -\into G_{h-k}(\ve)(T) \ dx\\\\
\dys\qquad+\into G_{h-k}(\uoe) \ dx 
\dys- \into S_{h+k}(\uoe-z_\nu ) \ dx
\end{array}
\]
Reasoning as in the proof of Lemma $2.1$ in \cite{po} we can easily show that both
\[
\begin{array}{l}
\dys \into S_{h+k}(\ve-\tkvn)(T) \ dx-\into G_{h-k}(\ve)(T) \ dx\geq 0,
\end{array}
\]
and
\[
\dys\into G_{h-k}(\uoe) \ dx - \into S_{h+k}(\uoe-z_\nu ) \ dx=\omega(\vare,\nu,h).
\]
Therefore we have proved that 
\[
\begin{array}{l}
\dys\int_{0}^{T}\langle\ve_{t}, w^\vare (1-\fde)\rangle\ dt \\\\
\dys\geq  \intq S_{h+k}(\ve-\tkvn) \  \frac{d\fde}{dt}\ dxdt\\\\
\dys\quad-\intq G_{h-k}(\ve) \ \frac{d\fde}{dt} \ dxdt +\omega(\vare,\nu,h),
\end{array}
\]
so, to conclude we have to check that 
\begin{equation}\label{toco}
\begin{array}{l}
\dys \intq S_{h+k}(\ve-\tkvn) \  \frac{d\fde}{dt}\ dxdt\\\\
\dys\quad-\intq G_{h-k}(\ve) \ \frac{d\fde}{dt} \ dxdt\geq \omega(\vare,\nu,h);
\end{array}
\end{equation}
actually, thanks to Proposition \ref{pro} and to the properties of $\tkvn$ we  have
\[
\begin{array}{l}
 \dys\intq S_{h+k}(\ve-\tkvn) \  \frac{d\fde}{dt}\ dxdt\\\\
\dys\quad-\intq G_{h-k}(\ve) \ \frac{d\fde}{dt} \ dxdt\\\\
\dys \geq\intq S_{h+k}(v-\tkv) \  \frac{d\fde}{dt}\ dxdt\\\\
\dys\quad-\intq G_{h-k}(v) \ \frac{d\fde}{dt} \ dxdt + \omega(\vare,\nu)\\\\
\dys= \intq F_{h} (v) \ \frac{d\fde}{dt}+ \omega(\vare,\nu),
\end{array}
\]
where $F_{h}(s) = S_{h+k}(s-T_k (s))-G_{h-k}(s)$; that is,   if $h>2k$,
\[
F_h (s)=
\begin{cases}
\quad\dys\int_{0}^{s-k}(T_{h+k}(\sigma)-T_{h-k}(\sigma))\ d\sigma &\text{if}\ \ s>h-k\\\\
 \qquad 0&\text{if}\ \ |s|\leq h-k \\\\
\quad\dys\int_{0}^{s+k}(T_{h+k}(\sigma)-T_{h-k}(\sigma))\ d\sigma &\text{if}\ \ s<-(h-k),
\end{cases} 
\]

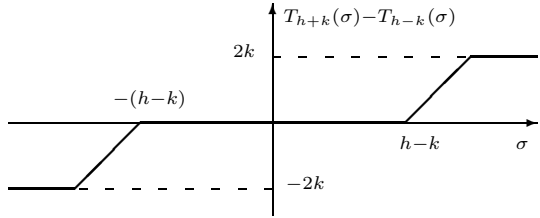
\begin{figure}[!ht]
\centerline{
\begin{picture}(200,80)(-100,-35)
\put(-100,0){\vector(1,0){200}} 
\put(0,-35){\vector(0,1){80}}
\thicklines
\put(-100,-25){\line(1,0){25}}
\put(-50,0){\line(-1,-1){25}}
\put(-50,0){\line(1,0){100}}
\put(50,0){\line(1,1){25}}
\put(75,25){\line(1,0){25}}
\thinlines
\put(-60.5,7){$\scriptstyle -(h-k)$}
\put(48,-10){$\scriptstyle h-k$}
\put(92,-10){$\scriptstyle \sigma$}
\put(4,38){$\scriptstyle T_{h+k}(\sigma)-T_{h-k}(\sigma) $}
\thinlines
 \multiput(0,25)(10,0){8}{\line(1,0){3}}
 \multiput(-3,-25)(-10,0){8}{\line(1,0){3}}
\put(-15,25){$\scriptstyle 2k$}
\put(5,-25){$\scriptstyle -2k$}
\end{picture}
}
\end{figure}

So, $F_{h}(v)$ converges almost everywhere to $0$ on $Q$ and, since $v\in \luq$, we can apply Lebesgue's dominated convergence theorem to conclude that \rife{toco} holds true.\\[1.0ex]
{\it Step $3$}.  Strong convergence of truncates.\newline
 Collecting together \rife{fn}, \rife{f}, and \rife{g} we have, taking again
$n>k$,
\begin{equation}\label{basi}
\begin{array}{l}
\dys\limsup_{\vare,\nu,n}\left(\dys\intq \a{u^\vare}\cdot\nabla \tkve\ dxdt\right.
\left.-\dys\intq \a{u^\vare}\cdot\nabla \tkvn\hnve\ dxdt\right)\leq 0,
\end{array}
\end{equation}
therefore, since using Egorov theorem and Proposition \ref{pro} we have
\[
\begin{array}{l}
\dys\intq \a{u^\vare}\cdot\nabla \tkvn\hnve\  dxdt
\dys=\intq \a{u}\cdot\nabla T_k(v)\ dxdt+\omega(\vare,\nu,n),
\end{array}
\]
then \rife{basi} implies \rife{basic}.

Now, recalling that
 \begin{equation}
\begin{array}{l}
\dys\intq\a{(\tkve+g^{\vare})}\cdot\nabla\tkve    \ dxdt\\\\
\dys=\intq\a{(\tkve+g^{\vare})}\cdot\nabla(\tkve +g^{\vare})
   \ dxdt\\\\ \dys- \intq\a{(\tkve+g^{\vare})}\cdot\nabla
g^{\vare}\ dxdt,
\end{array}
\end{equation}
 using   Fatou's lemma, and Proposition \ref{pro} we can easily conclude that
\begin{equation}
\begin{array}{l}
\dys\intq\a{(\tkve+g^{\vare})}\cdot\nabla(\tkve +g^{\vare})    \ dxdt\\\\
\dys=\intq\a{(\tkv+g )}\cdot\nabla(\tkv +g)    \ dxdt+\omega(\vare);
\end{array}
\end{equation}
Thus, being nonnegative, $\a{(\tkve+g^{\vare})}\nabla(\tkve +g^{\vare})$ actually converges to
$\a{(\tkv+g)}\nabla(\tkv +g)$ strongly in $L^1 (Q)$; hence, using assumption \rife{a1}
\[
\alpha|\nabla (\tkve+g^{\vare})|^p\leq\a{(\tkve+g^{\vare})}\cdot\nabla(\tkve +g^{\vare}),
\]
and so, by Vitali's theorem, recalling that $g^{\vare}$ strongly converges to $g $ in
$\psob$, we get
\[
\tkve\longrightarrow\tkv\ \ \ \text{strongly in}\ \ \psob.
\]
This concludes the proof of Theorem \ref{str}. 
\end{proof}\medskip

\setcounter{equation}{0}
\section{Existence of a renormalized solution}\label{7}

\ \ Now we are able to prove that problem \rife{base} has a renormalized solution.
\begin{proof}[Proof of Theorem \ref{esi}]

Let $S\in W^{2,\infty}(\re)$ such that $S'$ has a compact support as in Definition \ref{1}, and let $\varphi\in  C^{1}_{0}([0,T)\times\Omega)$; then the approximating solutions $\ue$ (and $\ve$) satisfy 
\begin{equation}
\begin{array}{l}
\dys\quad-\into S(\uoe)\varphi(0)\ dx -\intt\langle \varphi_t , S(\ve)\rangle\\\\
\dys \qquad +\intq S'(\ve) a(t,x,\nabla \ue)\cdot \nabla \varphi\ dxdt\\\\
\dys\qquad+\intq S''(\ve) a(t,x,\nabla \ue)\cdot \nabla \ve\ \varphi \ dxdt
\\\\ \dys=\intq S'(\ve) \varphi\ d\hat{\mu}^{\vare}+  \intq S'(\ve) \varphi\ d\lep - \intq S'(\ve) \varphi\ d\lem. 
\end{array}
\end{equation}
Thanks to Theorem \ref{str} all but the last term easily pass to the limit on $\vare$; actually the only terms that give some problems are the last two. We can write
\begin{equation}\label{aeb}
  \intq S'(\ve) \varphi\ d\lep =  \intq S'(\ve) \varphi\ \pdp d\lep + \intq S'(\ve) \varphi\ (1-\pdp) d\lep,
  \end{equation}
where $\pdp$ is defined as in Lemma \ref{acp}; thus
\begin{eqnarray*}
\left| \intq S'(\ve) \varphi\ (1-\pdp) d\lep\right|\leq C \intq   (1-\pdp)\ d\lep=\omega(\vare,\delta),
\end{eqnarray*}
while choosing $S'(\ve)\varphi\pdp$ in the formulation for $\ue$ one gets,
\begin{equation}
\begin{array}{l}
\dys  \intq S'(\ve) \varphi\pdp\ d\lep=-\intq S'(\ve) \varphi\pdp\ d\hat{\mu}^{\vare}+\intq S'(\ve) \varphi\pdp\ d\lem\\\\ \dys\quad-\intq (\varphi\pdp)_t\ S(\ve)\ dxdt\\\\
 \dys\qquad+\intq S'(\ve) a(t,x,\nabla \ue)\cdot \nabla (\varphi\pdp)\ dxdt \\\\
\qquad+\dys\intq S''(\ve) a(t,x,\nabla \ue)\cdot \nabla \ve\ \varphi \pdp\ dxdt; 
\end{array}
\end{equation}
now,  thanks to Proposition \ref{pro} and  the properties of $\pdp$, we readily have
$$
\intq S'(\ve) \varphi\pdp\ d\hat{\mu}^{\vare}=\omega(\vare,\delta),
$$
and, thanks to \rife{acp9}, 
$$
\left|\intq S'(\ve) \varphi\pdp\ d\lem\right|\leq C\intq\pdp\ d\lem=\omega(\vare,\delta)
$$
while, since $S(v)\in\psob\cap\liq$ and using \rife{acp6},
$$
\intq (\varphi\pdp)_t\ S(\ve)\ dxdt=\omega(\vare,\delta);
$$
moreover, since $ a(t,x,\nabla \ue)$ is strongly compact in $\luq$, $S'(\ve)$ is bounded, and $\pdp$ converges to zero in $\psob$ as $\delta$ goes to zero, we have
$$
\intq S'(\ve) a(t,x,\nabla \ue)\cdot \nabla (\varphi\pdp)\ dxdt=\omega(\vare, \delta),
$$
and, finally, using Theorem \ref{str} and the fact that $\nabla\ue = T_M (\ve) +g^\vare$ on the set $\{\ve\leq M\}$,
$$
\intq S''(\ve) a(t,x,\nabla \ue)\cdot \nabla \ve\ \varphi \pdp\ dxdt=\omega(\vare, \delta).
$$
Therefore, from \rife{aeb} we deduce
\begin{equation}
  \intq S'(\ve) \varphi\ d\lep = \omega(\vare).
  \end{equation}
  Analogously we can prove that
\begin{equation}
  \intq S'(\ve) \varphi\ d\lem = \omega(\vare).
  \end{equation}  
  Then $u$ satisfies equation \rife{eq1} with $\varphi\in C^{1}_{0}([0,T)\times\Omega)$;  now, an easy density argument shows that $u$ satisfies the same formulation with $\varphi \in \psob\cap \liq$ such that $\varphi_t \in \pw-1p'$, and $\varphi(T,x)=0$. 
  
  To prove the existence result it remains, then, to prove properties  \rife{rec} and \rife{recm}; so let us take $\hnv(1-\pdm)\varphi$ as test function in the formulation of $u$, where $\varphi\in\regq$. We obtain
  \begin{equation}\label{I}
\begin{array}{l}
\dys-\intq ((1-\pdm)\varphi)_{t} \  \ohnv\ dxdt \\\\
\dys\qquad+ \intq \hnv a(t,x,\nabla u )\cdot\nabla( (1-\pdm)\varphi) \ dxdt\\\\
\dys\quad=\intq  \hnv (1-\pdm)\varphi \ d\hat{\mu}_0\\\\
\qquad+\dys \frac{1}{n}\int_{\{n< v\leq 2n\}}  a(t,x,\nabla u )\cdot\nabla v \ (1-\pdm)\varphi\ dxdt\\\\
\dys\dys \quad-\frac{1}{n}\int_{\{-2n\leq v< -n\}}  a(t,x,\nabla u )\cdot\nabla v \ (1-\pdm)\varphi\ dxdt.
\end{array}
\end{equation}
 Now, recalling that $\ue$ is also a distributional  solution with datum $\mu^\vare$ we have
 \begin{equation}\label{II}
 \begin{array}{l}
 \dys\quad-\intq ((1-\pdm)\varphi)_t \ve\ dxdt +\intq a(t,x,\nabla\ue)\cdot\nabla((1-\pdm) \varphi)\ dxdt
 \\\\ \dys=
 \intq(1-\pdm)\varphi\ d\mhe +\intq(1-\pdm)\varphi\ d\lep-\intq(1-\pdm)\varphi\ d\lem,
 \end{array}
 \end{equation}
 for every $\varphi \in \regq$.
 
 Therefore, let us take the difference between \rife{I} and \rife{II}; we obtain
 \begin{ceqnarray}
&&
\quad-\intq ((1-\pdm)\varphi)_{t} \  \ohnv\ dxdt +\intq ((1-\pdm)\varphi)_t \ve\ dxdt \clabel{exa}{A}\\
&&
\dys\qquad+ \intq \hnv a(t,x,\nabla u )\cdot\nabla( (1-\pdm)\varphi) \ dxdt\clabel{exb}{B}\\
&&
\quad\dys-\intq a(t,x,\nabla\ue)\cdot\nabla((1-\pdm) \varphi)\ dxdt\clabel{exc}{C}\\
&&
\dys\quad-\intq  \hnv (1-\pdm)\varphi \ d\hat{\mu}_0+ \intq(1-\pdm)\varphi\ d\mhe \clabel{exd}{D}\\
&&
\dys \qquad+\frac{1}{n}\int_{\{-2n\leq v< -n\}}  a(t,x,\nabla u )\cdot\nabla v \ (1-\pdm)\varphi\ dxdt\clabel{exe}{E}\\
&&
\dys\qquad+\intq(1-\pdm)\varphi\ d\lep \clabel{exf}{F}\\
&&
\quad-\intq(1-\pdm)\varphi\ d\lem\clabel{exg}{G}\\
&&
\dys= \frac{1}{n}\int_{\{n< v\leq 2n\}}  a(t,x,\nabla u )\cdot\nabla v \ (1-\pdm)\varphi\ dxdt.\clabel{exh}{H}
\end{ceqnarray}
First of all, we easily have
\[
\rifer{exa}=\omega(\vare,n),
\]
and, thanks to Proposition \ref{pro},
\[
\rifer{exb}+\rifer{exc}=\omega(\vare,n).
\]
 Now, since $\hnv$ strongly converges to $1$ in $L^{p}(0,T;W^{1,p}(\Omega))$ (thanks to the estimate on the truncates of Proposition \ref{natur}, as we said before) and  then $\hnv (1-\pdm)\varphi $ converges to $ (1-\pdm)\varphi$ in $\psob$, we have
 \[
 \rifer{exd}=\omega(\vare,n). 
 \]
 Moreover,  thanks to \rife{acp11},
 \[
 \rifer{exg}=\omega(\vare,\delta)
 \]
 and  thanks to Lemma \ref{dazero} (see also Remark \ref{remesi}), and to Theorem \ref{str}, we have
 \[
 \rifer{exe}=\omega(n,\delta).
 \]
 
 Finally, using again Theorem \ref{str}, and Lemma \ref{dazero}, we have
 \[
 \rifer{exh}=\frac{1}{n}\int_{\{n< v\leq 2n\}}  a(t,x,\nabla u )\cdot\nabla v \ \varphi\ dxdt+\omega(n,\delta),
 \]
 while, by construction of $\lep$, we have
 \[
 \rifer{exf}=\intq\varphi\ d\msp +\omega(\vare,\delta).
 \]
 Putting together all the above results we obtain \rife{rec} for every $\varphi \in\regq$. Now, if $\varphi\in C^{\infty}(\overline{Q})$ we can split
 \begin{equation}\label{splitte}
 \begin{array}{l}
 \dys \frac{1}{n}\int_{\{n< v\leq 2n\}}  a(t,x,\nabla u )\cdot\nabla v \ \varphi\ dxdt\\\\=\dys \frac{1}{n}\int_{\{n< v\leq 2n\}}  a(t,x,\nabla u )\cdot\nabla v \ \varphi\pdp\ dxdt \\\\\qquad+\dys\frac{1}{n}\int_{\{n< v\leq 2n\}}  a(t,x,\nabla u )\cdot\nabla v \ \varphi(1-\pdp)\ dxdt,
 \end{array}
 \end{equation}
 and, thanks to what we proved before
 \begin{equation}\label{splitte1}
 \lim_{n\to\infty}\frac{1}{n}\int_{\{n< v\leq 2n\}}  a(t,x,\nabla u )\cdot\nabla v \ \varphi\pdp\ dxdt= \intq \varphi \ d\mu_{s}^{+} +\omega(\delta).
 \end{equation}
 
 On the other hand, reasoning as before, we are under the assumption of Lemma \ref{dazero} (see Remark \ref{remesi}), so we have
 $$
 \frac{1}{n}\int_{\{n< \ve\leq 2n\}}  a(t,x,\nabla \ue )\cdot\nabla \ve \ \varphi(1-\pdp)\ dxdt=\omega(\vare,n,\delta),
 $$ 
 that, gathered together with the strong convergence of truncates proved in Theorem \ref{str}, yields
 \begin{equation}\label{splitte2}
 \frac{1}{n}\int_{\{n< v\leq 2n\}}  a(t,x,\nabla u )\cdot\nabla v \ \varphi(1-\pdp)\ dxdt=\omega(n,\delta).
 \end{equation}
 Finally, putting together \rife{splitte}, \rife{splitte1}  and \rife{splitte2} we get \rife{rec} for every $\varphi \in  C^{\infty}(\overline{Q})$, and, reasoning by density, for every $\varphi \in  C(\overline{Q})$, which concludes the proof of \rife{rec}.  To obtain \rife{recm} we can reason as before using $\pdp$ in the place of $\pdm$ and viceversa, and  this concludes the proof of Theorem \ref{esi}.
\end{proof}\medskip

\setcounter{equation}{0}
\section{A partial uniqueness result and inverse maximum principle}\label{8}
\subsection{Uniqueness in the linear case}
   In this section we try to stress the fact that the notion of renormalized solution, as in the elliptic case, should be the right one
 to get uniqueness. As we said before if the datum $\mu$ belongs to $M_0 (Q)$ the renormalized solution turns out to be unique (see \cite{dpp}); the same happens for a general measure datum and $\uo\in \luo$ as initial condition, if the operator is linear, that is if
\begin{equation}\label{linear}
a(t, x,\xi)=M(t, x)\cdot\xi,
\end{equation}
where $M$ is a matrix with bounded, measurable entries, and satisfying the ellipticity assumption \rife{a1} (obviously with $p=2$). 
In fact we have
\begin{theorem}\label{tlinear}
Let M be as in \rife{linear}, $\mu\in M(Q)$, and $u_0 \in \luo$. Then the renormalized solution of problem
\begin{equation}\label{linbas}
\begin{cases}
    u_{t}- \mathrm{div} (M(t,x)\nabla u)
 =\mu & \text{in}\ (0,T)\times\Omega,\\
    u(0,x)=\uo & \text{in}\ \Omega,\\
 u(t,x)=0 &\text{on}\ (0,T)\times\partial\Omega,
  \end{cases}
\end{equation}
is unique.
\end{theorem}
\begin{proof}
We will prove this result by showing that a renormalized solution of problem \rife{linbas} is a solution in a 
\emph{duality sense};  uniqueness will follow immediately as in the elliptic case where the notion of duality solution was introduced and studied in \cite{s}.  So, let $w\in \sobl \cap C(\overline{Q})$ such that $w_t\in\l2h-1$ with $w(T)=0$ and choose $\ohnv$ and $w$ as test
functions in \rife{eq1}; we have
 \begin{ceqnarray}
 &&
 \dys\quad-\into \overline{H}_n (\uo)w(0)\ dx \clabel{8a}{A}\\
 &&
\quad-\dys\intt \left\langle w_t,\ohnv\right\rangle \ dt\clabel{8b}{B}\\ 
&&
\qquad\dys+\intq\ \hnv M(t,x)\nabla u\cdot\nabla w\ dxdt\clabel{8c}{C}\\ 
&&
\quad-\dys \frac{1}{n}\int_{\{n\leq v<2n\}}M(t,x)\nabla u\cdot\nabla v\ w \ dxdt\clabel{8d}{D}\\ 
&&
\qquad\dys+ \frac{1}{n}\int_{\{-2n< v\leq -n\}}M(t,x)\nabla u\cdot\nabla v\ w \ dxdt\clabel{8e}{E}\\ 
&&
=\dys\intq \hnv\ w\  d\mh,\clabel{8f}{F}
\end{ceqnarray}
and by properties \rife{rec} and \rife{recm} we readily obtain
\begin{equation}\label{recl}
\rifer{8a}+\rifer{8b}+\rifer{8c} -\rifer{8f}=\intq w \ d\mu_s + \omega(n),
\end{equation}
where $\mu_s =\msp-\msm$.

On the other hand if $\psi\in C^{\infty}_{0}(Q)$ we can choose $w$ as the solution of the \emph{parabolic
retrograde problem}
\begin{equation}\label{retro}
\begin{cases}
    -w_{t}- \mathrm{div} (M^{\ast}(t,x)\nabla w)
 =\psi & \text{in}\ (0,T)\times\Omega,\\
    w(T,x)=0 & \text{in}\ \Omega,\\
 w(t,x)=0 &\text{on}\ (0,T)\times\partial\Omega,
  \end{cases}
\end{equation}
where $M^{\ast}(t,x)$ is the transposed matrix of $M(t,x)$;
now, since both $\hnv$ and $g$ are  good test functions for this problem, and recalling that
 $\ohnv$ converges to $v$ in $L^1(Q)$ while $\overline{H}_n (\uo)$ converges to $\uo$, we have
\begin{equation}\label{recll}
\begin{array}{l}
\rifer{8a}+\rifer{8b}+\rifer{8c} -\rifer{8f}=\dys-\into \overline{H}_n (\uo)w(0)\ dx \\\\
\quad-\dys\intt \left\langle w_t,\ohnv\right\rangle \ dt
\dys+\intq\ \nabla \ohnv \cdot M^\ast(t,x) \nabla w\ dxdt\\\\
\dys\quad-\intq \hnv\ w\  d\mh
\dys+\intq \hnv\nabla g\cdot M^\ast(t,x) \nabla w\ dxdt\\\\
\dys=-\into \overline{H}_n (\uo)w(0)\ dx+\intq \ohnv\psi\ dxdt\\\\
\dys\quad-\intq \hnv\ w\  d\mh +\intq \hnv\nabla g\cdot M^\ast(t,x) \nabla w\ dxdt
\dys\\\\\dys=-\into \uo w(0)\ dx+\intq v\psi\ dxdt\\\\
\dys\quad-\intq  w\  d\mh +\intq \nabla g\cdot M^\ast(t,x) \nabla w\ dxdt +\omega(n)\\\\
\dys=-\into \uo w(0)\ dx+\intq v\psi\ dxdt
\dys-\intq  w\  d\mh +\intq  g \ \psi dxdt \\\\
\dys\qquad+\intt\langle w_t ,g \rangle \ dt+\omega(n)\\\\
\dys=-\into \uo w(0)\ dx+\dys\intq u\psi\ dxdt
\dys-\intq  w\  d \mu_{0}  +\omega(n);
\end{array}
\end{equation}
hence, comparing \rife{recl} and \rife{recll}, we obtain
\begin{equation}\label{iffe}
-\into \uo w(0)\ dx+\intq u\psi\ dxdt=\intq w\ d\mu,
\end{equation}
for every $\psi\in C^{\infty}_{0}(Q)$. Therefore, since if $u_1$ and $u_2$ satisfy \rife{iffe}, then
$$
\intq (u_1 -u_2 )\psi\ dxdt=0
$$
for every $\psi\in C^{\infty}_{0}(Q)$, and so $u_1 = u_2$, $u$ is the unique solution of \rife{linbas}.
\end{proof}\medskip

\subsection{Inverse maximum principle for general parabolic operators}

In the elliptic case, an easy consequence of the definition and existence of  a renormalized solution  (see \cite{dmop}) is the so called  \emph{Inverse maximum principle} for general monotone operators proved independently in
\cite{dp} in the model case of the Laplace operator. This result has a large number of interesting  
applications; for instance it allows  to prove a generalized \emph{Kato's inequality} when $\Delta u$ is a measure (see \cite{bp}). 
In the same way for parabolic equations,  a straightforward consequence of Definition \ref{1} and Theorem \ref{esi}, using again the notation $v=u-g$, is the 
  following result where, for technical reasons we must make a stronger assumption on $g$.
\begin{theorem}[Parabolic "inverse" maximum principle]\label{imp}
Let $\mu\in M(Q)$, and suppose that there exists $g\in\psob\cap\liq$ such that $\mu$ can be decomposed as in \rife{app0};  let  $u$ be the renormalized solution of problem \rife{base}. Then,   if $u\geq 0$, we have
$\mu_s \geq 0$.
\end{theorem}
\begin{remark}\label{nove}
Notice that obviously the result has an easy \emph{nonpositive} counterpart and that Theorem \ref{imp} applies, in particular, for purely singular data.  Also observe  that  the stronger assumption on $g$  is rather technical and relies on the fact that we are not able to prove that, in the \emph{decomposition} Theorem \ref{cap2}, $g$ can be chosen to be  bounded, this question being still  an open problem. Finally notice that the sign assumption on $u$ in Theorem \ref{imp} can be relaxed; actually, because of the reconstruction property \rife{recm}, the same result holds true even  if $u$ is only supposed to be bounded from below (or from above in the nonpositive analogue).
\end{remark}

\end{document}